\documentclass[12pt, reqno]{amsart}
\usepackage{amssymb}
\usepackage{bbm}
\usepackage{amsthm}
\usepackage{nccmath}
\usepackage[bookmarks=false]{hyperref}
\usepackage{etoolbox}
\usepackage{array}
\usepackage{xparse}
\usepackage{enumitem}
\usepackage{tikz}
\usepackage{mathtools}
\usepackage[font=scriptsize]{caption}
\usepackage[capitalise]{cleveref}

\crefname{equation}{}{}
\Crefname{equation}{}{}

\calclayout

\newtheorem{thm}{Theorem}[section]
\newtheorem{theorem}[thm]{Theorem}
\newtheorem{lem}[thm]{Lemma}
\newtheorem{lemma}[thm]{Lemma}
\newtheorem{prop}[thm]{Proposition}
\newtheorem{cor}[thm]{Corollary}

\theoremstyle{definition}

\theoremstyle{remark}

\newcommand{\thmref}[1]{Theorem~\ref{#1}}

\newcommand{\lemref}[1]{Lemma~\ref{#1}} 
\newcommand{\propref}[1]{Proposition~\ref{#1}}

\newcommand*{\dis}{\displaystyle}

\newcommand*{\qq}{\qquad}

\newcommand*{\tx}[1]{\text{#1}}

\newcommand*{\lamb}{\lambda}
\newcommand*{\del}{\delta}
\newcommand*{\ep}{\epsilon}

\newcommand*{\suchthat}{\, \middle| \,}

\newcommand*{\ftil}{\widetilde{f}}
\newcommand*{\util}{\widetilde{u}}

\newcommand*{\gtil}{\widetilde{g}}

\newcommand*{\Ntil}{\widetilde{N}}

\newcommand*{\Pminus}{P_{-}}

\newcommand*{\myoverline}[3]{\mkern -#1mu\overline{\mkern#1mu#3\mkern#2mu}\mkern -#2mu}	

\newcommand*{\Zbar}{\myoverline{-3}{0}{\Z}}
\newcommand*{\zbar}{\myoverline{-2}{0}{\z}}

\newcommand*{\ubar}{\myoverline{0}{0}{u}}

\newcommand*{\half}{\frac{1}{2}}

\newcommand*{\Rsp}{\mathbb{R}}
\newcommand*{\Csp}{\mathbb{C}}

\newcommand*{\Zsp}{\mathbb{Z}}

\newcommand*{\Scal}{\mathcal{S}}
\newcommand*{\Dcal}{\mathcal{D}}
\newcommand*{\Lcal}{\mathcal{L}}

\newcommand*{\Ltwo}{L^2}

\newcommand*{\Linfty}{L^{\infty}}
\newcommand*{\Hhalf}{\dot{H}^\half}

\newcommand*{\al}{\alpha}
\newcommand*{\ap}{{\alpha'}}

\newcommand*{\eqdef}{\mathrel{\mathop:}=}
\newcommand*{\diff}{\mathop{}\! d}

\newcommand*{\compose}[1]{\circ{#1}}

\newcommand*{\Hil}{\mathbb{H}}
\newcommand*{\Hcal}{\mathcal{H}}

\newcommand*{\Id}{\mathbb{I}}
\newcommand*{\Imag}{\tx{Im}}
\newcommand*{\Real}{\tx{Re}}
\newcommand*{\Jdel}{J_\delta}

\newcommand*{\sgn}{sgn}

\newcommand*{\grad}{\nabla}
\newcommand*{\Dt}{D_t}

\newcommand*{\pt}{\partial_t}
\newcommand*{\ps}{\partial_s}
\newcommand*{\px}{\partial_x}

\newcommand*{\pap}{\partial_\ap}

\newcommand*{\pal}{\partial_\al}

\newcommand*{\Dal}{D_{\al}}
\newcommand*{\Dap}{D_{\ap}}

\newcommand*{\Dabs}{\abs{D}}

\newcommand*{\Fcal}{\mathcal{F}}

\newcommand*{\Asigma}{A_{\sigma}}

\newcommand*{\bvar}{b}

\newcommand*{\avar}{a}

\newcommand*{\h}{h}

\newcommand*{\hal}{\h_\al}

\newcommand*{\hinv}{\h^{-1}}

\newcommand*{\thvar}{\theta}
\newcommand*{\Th}{\Theta}

\newcommand*{\z}{z}

\newcommand*{\zal}{\z_\al}

\newcommand*{\zalbar}{\zbar_\al}

\newcommand*{\zalabs}{\abs{\zal}}

\newcommand*{\zt}{\z_t}

\newcommand*{\ztbar}{\zbar_t}

\newcommand*{\Z}{Z}

\newcommand*{\Zap}{\Z_{,\ap}}

\newcommand*{\Zapbar}{\Zbar_{,\ap}}

\newcommand*{\Zapabs}{\abs{\Zap}}

\newcommand*{\Zt}{\Z_t}

\newcommand*{\Ztbar}{\Zbar_t}

\newcommand*{\Ztap}{\Z_{t,\ap}}

\DeclarePairedDelimiter{\oldbrac}{\lparen}{\rparen}			
\NewDocumentCommand{\brac}{ s o m }{						
	\IfBooleanT{#1}{
  		\IfValueT{#2}{\oldbrac[#2]{#3}}
		\IfValueF{#2}{\oldbrac{#3}} 
	}
	\IfBooleanF{#1}{
  		\IfValueT{#2}{\PackageError{mypackage}{Incorrect use of brac. Insert star}{}}
		\IfValueF{#2}{\oldbrac*{#3}} 
	}		
}

\DeclarePairedDelimiter\oldcbrac{\lbrace}{\rbrace}				
\NewDocumentCommand{\cbrac}{ s o m }{					
	\IfBooleanT{#1}{
  		\IfValueT{#2}{\oldcbrac[#2]{#3}}
		\IfValueF{#2}{\oldcbrac{#3}} 
	}
	\IfBooleanF{#1}{
  		\IfValueT{#2}{\PackageError{mypackage}{Incorrect use of cbrac. Insert star}{}}
		\IfValueF{#2}{\oldcbrac*{#3}} 
	}		
}

\DeclarePairedDelimiter\oldsqbrac{\lbrack}{\rbrack}				
\NewDocumentCommand{\sqbrac}{ s o m }{					
	\IfBooleanT{#1}{
  		\IfValueT{#2}{\oldsqbrac[#2]{#3}}
		\IfValueF{#2}{\oldsqbrac{#3}} 
	}
	\IfBooleanF{#1}{
  		\IfValueT{#2}{\PackageError{mypackage}{Incorrect use of sqbrac. Insert star}{}}
		\IfValueF{#2}{\oldsqbrac*{#3}} 
	}		
}

\DeclarePairedDelimiter\oldabrac{\langle}{\rangle}				
\NewDocumentCommand{\abrac}{ s o m }{					
	\IfBooleanT{#1}{
  		\IfValueT{#2}{\oldabrac[#2]{#3}}
		\IfValueF{#2}{\oldabrac{#3}} 
	}
	\IfBooleanF{#1}{
  		\IfValueT{#2}{\PackageError{mypackage}{Incorrect use of abrac. Insert star}{}}
		\IfValueF{#2}{\oldabrac*{#3}} 
	}		
}

\DeclarePairedDelimiter{\oldabs}{\lvert}{\rvert}
\NewDocumentCommand{\abs}{ s o m }{						
	\IfBooleanT{#1}{
  		\IfValueT{#2}{\oldabs[#2]{#3}}
		\IfValueF{#2}{\oldabs{#3}} 
	}
	\IfBooleanF{#1}{
  		\IfValueT{#2}{\PackageError{mypackage}{Incorrect use of abs. Insert star}{}}
		\IfValueF{#2}{\oldabs*{#3}} 
	}		
}

\DeclarePairedDelimiterX{\oldnorm}[1]{\lVert}{\rVert}{#1}
\NewDocumentCommand{\norm}{ s o o m }{					
	\IfValueT{#2} {
		\IfBooleanT{#1}{
  			\IfValueT{#3}{\oldnorm[#2]{#4}_{#3}}
			\IfValueF{#3}{\oldnorm{#4}_{#2}} 
		}
		\IfBooleanF{#1}{
  			\IfValueT{#3}{\PackageError{mypackage}{Incorrect use of norm. Insert star}{}}
			\IfValueF{#3}{\oldnorm*{#4}_{#2}} 
		}
	}
	\IfValueF{#2} {
		\IfBooleanT{#1}{\oldnorm{#4}}	
		\IfBooleanF{#1}{\oldnorm*{#4}}		
	}	
}

\makeatletter
\def\black@#1{%
    \noalign{%
        \ifdim#1>\displaywidth
            \dimen@\prevdepth
            \nointerlineskip
            \vskip-\ht\strutbox@
            \vskip-\dp\strutbox@
            \vbox{\noindent\hbox to \displaywidth{\hbox to#1{\strut@\hfill}}}%
            \prevdepth\dimen@
        \fi
    }%
}
\makeatother

\makeatletter
\renewcommand{\tocsection}[3]{%
  \indentlabel{\@ifnotempty{#2}{\bfseries\ignorespaces#1 #2\quad}}\bfseries#3}
\renewcommand{\tocsubsection}[3]{%
  \indentlabel{\@ifnotempty{#2}{\ignorespaces#1 #2\quad}}#3}

\newcommand\@dotsep{4.5}
\def\@tocline#1#2#3#4#5#6#7{\relax
  \ifnum #1>\c@tocdepth 
  \else
    \par \addpenalty\@secpenalty\addvspace{#2}%
    \begingroup \hyphenpenalty\@M
    \@ifempty{#4}{%
      \@tempdima\csname r@tocindent\number#1\endcsname\relax
    }{%
      \@tempdima#4\relax
    }%
    \parindent\z@ \leftskip#3\relax \advance\leftskip\@tempdima\relax
    \rightskip\@pnumwidth plus1em \parfillskip-\@pnumwidth
    #5\leavevmode\hskip-\@tempdima{#6}\nobreak
    \leaders\hbox{$\m@th\mkern \@dotsep mu\hbox{.}\mkern \@dotsep mu$}\hfill
    \nobreak
    \hbox to\@pnumwidth{\@tocpagenum{\ifnum#1=1\bfseries\fi#7}}\par
    \nobreak
    \endgroup
  \fi}
\AtBeginDocument{%
\expandafter\renewcommand\csname r@tocindent0\endcsname{0pt}
}
\def\l@subsection{\@tocline{2}{0pt}{2.5pc}{5pc}{}}
\makeatother

\makeatletter
 \def\@testdef #1#2#3{%
   \def\reserved@a{#3}\expandafter \ifx \csname #1@#2\endcsname
  \reserved@a  \else
 \typeout{^^Jlabel #2 changed:^^J%
 \meaning\reserved@a^^J%
 \expandafter\meaning\csname #1@#2\endcsname^^J}%
 \@tempswatrue \fi}
\makeatother

\makeatletter
\newcommand*{\rom}[1]{\expandafter\@slowromancap\romannumeral #1@}
\makeatother

\makeatletter
\patchcmd{\@sect}{\@addpunct.}{}{}{}
\patchcmd{\subsection}{-.5em}{1em}{}{}
\makeatother

\AtBeginDocument{%
   \def\MR#1{}
}

\begin{document}

\title[Self-Similar Solutions to the Hele-Shaw Problem]{Self-Similar Solutions to the Hele-Shaw Problem with Surface Tension}
\author{Siddhant Agrawal and Neel Patel}
\address{\parbox{\linewidth}{Department of Mathematics \\
University of Colorado Boulder \\
Boulder, Colorado, USA 80309\smallskip}}
\email{Siddhant.Agrawal@colorado.edu}
\address{\parbox{\linewidth}{Department of Mathematics and Statistics \\
University of Maine \\
Orono, Maine, USA 04469\smallskip}}
\email{neel.patel@maine.edu}
%
\begin{abstract}
We consider the Hele-Shaw problem with surface tension in an infinite domain. We prove the existence of a family of self-similar solutions. At $t=0$, these solutions have a corner of angle $\theta$ with $ 0 < |\theta - \pi| \ll 1$, and for $t>0$, the solutions are smooth.
\end{abstract}
%

\maketitle
\tableofcontents

\section{Introduction}

Fluid flow in porous media is described by Darcy's law. In two dimensions, Darcy's law is
\begin{align}\label{eq:Darcy}
 \frac{\mu(x,t)}{\kappa(x,t)} u(x,t) = -\grad P(x,t).
\end{align}
where $u(x,t)$ is the velocity of the fluid at position $x\in\mathbb{R}^{2}$ and time $t\geq 0$, $P(x,t)$ is the pressure, $\mu(x,t)$ is the fluid viscosity and $\kappa(x,t)$ is the permeability of the media. This type of fluid flow is replicated experimentally in a Hele-Shaw cell. The Hele-Shaw problem describes the dynamics of the boundary $\partial\Omega(t)$ of a single fluid in porous media or in a Hele-Shaw cell. If there are two fluids and a gravitational force, the dynamics of the boundary are described by the Muskat problem. In both of these problems, there is extensive literature on the question of local and global well-posedness. For references, see \cite{Am04, CoCoGa11, EsMa11, CaCoFeGaGo13, CoCo16, ChGr16, CoGa17, PaSt17, GaGaPa19, AlLa20, AlMeSm20, NgPa20, CoLa21, AlNg21b, KeNgXu22, DoGaNg23, GaGaPaSt23, MaMa23, La24, GaGaPaSt25, ScTuTu25} and further references therein.

In this paper, we are interested in the dynamics of interfaces with corners in a Hele-Shaw cell. We will consider the Hele-Shaw problem with surface tension and no gravity for a single incompressible fluid of constant density on an infinite domain, with no bottom. Denote the region occupied by the fluid as $\Omega(t) \subset \mathbb{R}^{2}$. Normalizing the fluid characteristics $\mu = \kappa = 1$, the Hele-Shaw problem is
\begin{align}\label{eq:HeleShaw}
\begin{aligned}
	u &= - \grad P  \qq \tx{ in } \Omega(t)\\
	\nabla\cdot u &= 0 \qq \qq \tx{ in } \Omega(t).
\end{aligned}
\end{align}
Due to surface tension, the pressure on the boundary is given by 
\begin{align}\label{eq:surfacetension}
	P = -\sigma \partial_{s}\thvar \qq\text{ on } \partial\Omega(t)
\end{align}
where $\thvar$ is the angle of interface with respect to the positive $x -axis$, $\ps$ is the arc length derivative, and $\sigma \geq 0 $ is the coefficient of surface tension. 

Recently, there has been a lot of interest in studying the Hele-Shaw and Muskat problem for interfaces with corners. In the case of zero surface tension for the Hele-Shaw problem, we see different dynamics for corners of angle less than $\pi/2$ and for corners of angle greater than $\pi/2$. Corners of angle less than $\pi/2$ remain rigid for a short time, which is called the waiting time phenomenon. On the other hand, corners of angle greater than $\pi/2$ instantaneously smoothen out. See \cite{KiLa95, ChKi06, ChJeKi07, ChJeKi09, Sa10, BaVa14, AgPaWu23, AlKo23}. For the Muskat problem without surface tension, for angles greater $\pi/2$, the interface smoothens out instantaneously \cite{Ca19, GaGoHaPa24}. For angles less than $\pi/2$, in contrast with the Hele-Shaw problem, it is expected that the interface smoothens out instantaneously \cite{GaGoNgPa22}. In the case of zero surface tension, we would also like to mention that self-similar solutions that initially have a corner have also been studied in both the Hele-Shaw and Muskat problems \cite{KiLa95, GaGoNgPa22, Na25}.

We note that all of the above papers concerning corner dynamics have been in the case of zero surface tension. Recently for the Muskat problem, global wellposedness and instantaneous smoothing for initial interfaces with small Lipschitz constant has been proved in \cite{ChHuNg24}. However, the Hele-Shaw problem is structurally different from the Muskat problem because it is a one-phase and not a two-phase problem. To the best of our knowledge, there are no rigorous results on the study of dynamics of corners with surface tension for the Hele-Shaw problem, although there are numerical or asymptotic studies \cite{GaMeVi06, MeSaPaVi07}. There have been some work regarding the contact angle problem with surface tension \cite{AvBa05, BoCaGa25}; however, having a corner in the interior of an interface with surface tension introduces significant new difficulties and has not been studied rigorously to date.

In this paper, we consider the Hele Shaw equation with surface tension and prove the existence of self similar solutions. The solutions we construct have a corner at $t=0$ and are smooth for all $t>0$. The angle of the corner of the fluid domain $\theta_0$ at $t=0$ satisfies $\abs{\theta_0 - \pi} < \ep_0$ for some small $\ep_0>0$. We construct these self-similar solutions in conformal coordinates. The precise result is stated in Theorem \ref{thm:mainselfsimilarresult}.

This paper provides the first proof of the existence of solutions to the Hele-Shaw problem with surface tension where the interface has a corner at $t=0$. Because of surface tension, the equation we obtain for our self-similar solution is a third order nonlocal equation of elliptic type. The fact that the equation is of third order makes it significantly more challenging to study as compared to the first order equations for self-similar solutions of previous results. A significant advantage of the conformal coordinate approach of this paper is that we derive a very explicit form of the self-similar equation. In particular, the only nonlocal operator in our self-similar equation is the Hilbert transform. Moreover, from our equation, we derive a linear equation that we solve exactly. This solution gives an excellent approximation to the self-similar solution and this approximation also matches with the linear analysis in \cite{GaMeVi06}. On the other hand, a price that we have to pay for using conformal coordinates is that our self-similar equation has coefficients that grow at infinity. This difficulty makes our linear analysis much more technical.

We believe that the techniques of this paper can help with addressing the existence of solutions to other free boundary problems (e.g. the wave wave equation) with surface tension where the initial interface has a corner.

The paper is organized as follows. In Section \ref{sec:prelim}, we introduce the notation used in the paper. In particular, we define the conformal mapping and derive the interface equation in conformal coordinates. In Section \ref{sec:mainresults}, we define the notion of self-similar solutions and state the main results of this paper. Next, in Section \ref{sec:selfsimilareq}, we derive the self-similar equation in conformal coordinates. In Section \ref{sec:linearanalysis}, we study the corresponding linear problem and then, in Section \ref{sec:fullequation}, we apply our linear analysis to solve the self-similar equation.

\section{Notation and Preliminaries}\label{sec:prelim}

In this section and the next we will assume some regularity and decay properties of the solution to derive the main evolution equations, which we then solve in later sections to construct solutions. Near the end of this paper, we will show that these constructed solutions indeed satisfy the  assumptions assumed here and satisfy the original Hele Shaw equation.

In this paper we will follow the notation used in \cite{Ag21, AgPaWu23}. We write $a \lesssim b$ if there exists a universal constant $C>0$ so that $a \leq Cb$. We write $a \lesssim_M b$ if there exists a constant $C(M)$ depending only on $M$ so that $a \leq C(M)b$. Similar definitions for $a\lesssim_{M_1, M_2} b$ etc. 
For two complex valued functions $f$ and $g$, we write $f(z) \simeq g(z)$ as $z \to z_0$ if $\lim_{z \to z_0} \frac{f(z)}{g(z)} = 1$. For $z \in \Csp$ and $s>0$, we define $z^s = e^{s\ln(z)}$ where the branch cut for $\ln(z)$ is the positive imaginary axis. All singular integrals will be understood in the principle value sense, and we will frequently suppress writing p.v. in front of the integrals. 

The Fourier transform for a function $f: \Rsp \to \Csp$ is defined as
 \begin{align*}
\hat{f}(\xi) = \frac{1}{\sqrt{2\pi}}\int_\Rsp e^{-ix\xi}f(x) \diff x
\end{align*}
We will denote by $\Dcal(\Rsp)$ the space of smooth functions on $\Rsp$ with compact support and $\Scal(\Rsp)$ is the Schwartz space of rapidly decreasing functions.  $\Dcal'(\Rsp)$ and $\Scal'(\Rsp)$ are the space of distributions and tempered distributions on $\Rsp$ respectively. A Fourier multiplier with symbol $a(\xi)$ is the operator $T_a$ defined formally by the relation $\dis \widehat{T_a{f}}(\xi) = a(\xi)\hat{f}(\xi)$. For $s \geq 0$ the operators $\Dabs^s $ and $\langle D \rangle^s$ are defined as the Fourier multipliers with symbols $\abs{\xi}^s$ and $(1 + \abs{\xi}^2)^{\frac{s}{2}}$ respectively. 
For $s \geq 0$, the Sobolev space $H^s(\Rsp)$ and homogenous Sobolev space $\dot{H}^s(\Rsp)$ are the space of all $f \in \Scal'(\Rsp)\cap L^1_{loc}(\Rsp)$ with
\begin{align*}
\norm[H^s]{f} = \norm*[\Ltwo(\diff x)]{\langle D \rangle^s f} < \infty \\
\norm[\dot{H}^s]{f} = \norm*[\Ltwo(\diff x)]{\Dabs^s f} < \infty
\end{align*}
Note that for $s>0$, constant functions are in the homogenous Sobolev space $\dot{H}^s(\Rsp)$. We define the Hilbert transform as
\begin{align}\label{eq:HilbertP}
\Hil f(x) := \frac{1}{i \pi} \text{p.v.} \int_{\Rsp} \frac{f(y)}{x-y} \diff y
\end{align}
Observe that with this definition of the Hilbert transform we have $\sqrt{-\Delta} = \Dabs = i \Hil \px$ and the Fourier multiplier of $\Hil$ is $- sgn(\xi)$. The Hilbert transform defined above satisfies the following property.
\begin{lem}[\cite{Ti86}]\label{lem:Hilhol}
Let $1<p<\infty$ and let $F(z)$ be a holomorphic function in the lower half plane with $F(z) \to 0$ as $z\to \infty$. Then the following are equivalent:
\begin{enumerate}
\item $\dis \sup_{y<0} \norm[p]{F(\cdot + iy)} < \infty$
\item $F(z)$ has a boundary value $f$, non-tangentially almost everywhere with $f \in L^p(\Rsp)$ and $\Hil(f) = f$. 
\end{enumerate}
\end{lem}

From now on compositions of functions will always be in the spatial variables and we write $f = f(\cdot,t), g = g(\cdot,t), f \compose g(\cdot,t) :=  f(g(\cdot,t),t)$. Let $[A,B] := AB - BA$ be the commutator of the operators $A$ and $B$. We will denote the lower half plane by $\Pminus = \cbrac{(x,y) \in \Rsp^2 \suchthat y<0}$. In this section we will work with both Lagragian and conformal coordinates and we will use the variables $\alpha$ and $\ap$ respectively for these different coordinate systems.

Let the interface be parametrized in Lagrangian coordinates by $\z(\cdot,t): \Rsp \to \partial\Omega(t)$ satisfying $\z_{\al}(\al,t) \neq 0 $ for all  $\al \in \Rsp$. Hence $\zt(\al,t) = u(\z(\al,t),t)$ is the velocity of the fluid on the interface. From \eqref{eq:HeleShaw}, we have on the boundary
\begin{align}\label{eq:HeleShaworigbdry}
\zt = -\grad P
\end{align}
As ${\frac{\zal}{\zalabs}(\al,t) = e^{i\thvar(\al,t)}}$ and $\frac{1}{\zalabs}\pal$ is the arc length derivative in Lagrangian coordinates, the pressure \eqref{eq:surfacetension} on the boundary can be rewritten as
\begin{align*}
P(\z(\al,t),t) = i\sigma \frac{1}{\zal}\pal\frac{\zal}{\zalabs} (\al,t)
\end{align*}
Now 
\begin{align*}
\grad P = n \frac{\partial P}{\partial n} + \tau \frac{\partial P}{\partial \tau}
\end{align*}
where
\begin{align*}
& \tau = \frac{\zal}{\zalabs} = e^{i\thvar} = \tx{unit tangent vector} \\
& n = i\frac{\zal}{\zalabs} = ie^{i\thvar} = \tx{unit outward normal vector}
\end{align*}
Let ${\avar} = - \frac{1}{\abs{\zal}}\frac{\partial P}{\partial n}$ and let $\Dal = \frac{1}{\zal}\pal$. Note that $\frac{\partial}{\partial \tau} = \frac{1}{\zalabs}\pal$ and so
\begin{align*}
\grad P = -i{\avar}\zal + \frac{\zal}{\zalabs} \frac{1}{\zalabs}\pal \brac{ i\sigma \frac{1}{\zal}\pal\frac{\zal}{\zalabs} } = -ia\zal + \frac{1}{\zalbar}\pal \brac{ i\sigma \frac{1}{\zal}\pal\frac{\zal}{\zalabs} }
\end{align*}
Plugging this into \eqref{eq:HeleShaworigbdry} and taking a conjugate we get
\begin{align}\label{eq:ztbar}
\ztbar = -i{\avar}\zalbar -i\sigma\Dal^2\brac{\frac{\zal}{\zalabs}}
\end{align}

Let $\Psi(\cdot,t): \Pminus \to  \Omega(t)$ be conformal maps such that $\Psi$ is differentiable in time with $\lim_{\z\to \infty} \Psi_t(\z,t) =0$. Note that there is an ambiguity in the choice of $\Psi$ due to the scaling and translation symmetries in a conformal mapping, however they will be fixed once we impose the self similar restriction in the next section. Let $\Phi(\cdot,t):\Omega(t) \to \Pminus $ be the inverse of the map $\Psi(\cdot,t)$ and define $\h(\cdot,t):\Rsp \to \Rsp$ as
\begin{align}\label{eq:h}
    \h(\alpha,t) \eqdef \Phi(\z(\alpha,t),t).
\end{align}
For $\ap \in \Rsp$, we define the map $\hinv(\cdot,t): \Rsp \to \Rsp$ as the inverse of $h(\cdot,t)$ and hence
\begin{align*}
    \h(\hinv(\ap,t),t) = \ap.
\end{align*}
Precomposing $\z(\al,t)$ with $\hinv$, we define
\begin{align*}
    Z(\ap,t) & \eqdef z\compose \hinv (\ap,t), \\
    \Zt(\ap,t) &\eqdef \zt\compose \hinv (\ap,t)
\end{align*}
as the re-parametrizations of the position and velocity on the interface under the Riemann mapping. Denote the derivatives in $\ap$ by
\begin{align*}
    \Zap(\ap,t) = \pap \Z(\al,t) \quad \text{  and  } \quad  \Ztap(\ap,t) = \pap \Zt(\al,t).
\end{align*}
With the above conventions, we have $\Psi\compose \h(\al,t) = \z(\al,t)$, and hence, it follows that
\begin{align*}
\Psi(\al;t) &= \Psi(\h(\hinv(\al,t);t);t)\\
&= \z(\hinv(\al,t);t)\\
&= \Z(\al,t).
\end{align*}
Thus, by the chain rule,
\begin{align}
\Zap(\ap,t) \eqdef \pap \Z(\ap,t) & = \pap (\z\compose \hinv(\al,t)) \nonumber\\
&= \zal\compose \hinv (\al,t) \cdot \frac{1}{\hal\compose \hinv(\al,t)}. \label{eq:Zaphapzap}
\end{align}

Now let $\Dap = \frac{1}{\Zap} \pap$ and let $\Asigma$ be the real valued function defined as
\begin{align}\label{eq:Asigma}
\Asigma = \Zapabs^2(a \hal)\compose \hinv .
\end{align}
Hence precomposing \eqref{eq:ztbar} with $\hinv$ we get
\begin{align}\label{eq:Ztbartemp1}
\Ztbar = -i \frac{\Asigma}{\Zap} -i\sigma\Dap^2\frac{\Zap}{\Zapabs}
\end{align}
Multiplying by $\Zap$ and rearranging gives us
\begin{align*}
\Asigma = i\Ztbar\Zap  -\sigma\pap\Dap\frac{\Zap}{\Zapabs}
\end{align*}
As the velocity is both divergence and curl free, this implies that $\ubar$ is holomorphic in $\Omega(t)$. Hence $\ubar\compose \Psi$ is holomorphic on $\Pminus$ and so $\Ztbar$ is the boundary value of a holomorphic function on $\Pminus$. Hence $\Ztbar\Zap$ is the boundary value of a holomorphic function. We now assume that both $\Ztbar\Zap$ and $\sigma\pap\Dap\frac{\Zap}{\Zapabs}$ go to $0$ as $\abs{\ap} \to \infty$. Now applying $\Real(\Id - \Hil)$ to the above equation gives us
\begin{align}\label{eq:Agsigma}
\Asigma =  - \sigma\Real(\Id - \Hil)\pap\Dap\frac{\Zap}{\Zapabs} = \sigma\pap\Hil\Dap\frac{\Zap}{\Zapabs}
\end{align}
Plugging this in \eqref{eq:Ztbartemp1} gives
\begin{align*}
\Ztbar  = -i\sigma\Dap(\Id + \Hil)\Dap\frac{\Zap}{\Zapabs}
\end{align*}
Let $\Th = -i(\Id + \Hil)\Dap\frac{\Zap}{\Zapabs}$. Note that $\Real\Th  $ equals the curvature of the interface and $\Th $ is the boundary value of a holomorphic function. Using this variable we get
\begin{align}\label{eq:Ztbarmain}
\Ztbar = \sigma\Dap\Th
\end{align}

We define the material on the boundary $\Dt$ as
\begin{align*}
 \Dt = \pt + \bvar \pap \quad \tx{ where } \bvar = \h_t \compose \hinv
\end{align*}
Note that $\bvar$ is real valued. With this definition we have $\Dt \Z = \Zt$ and more generally $\Dt (f(\cdot,t)\compose \hinv) = (\pt f(\cdot,t)) \compose \hinv$ or equivalently $\pt (F(\cdot,t)\compose \h) = (\Dt F(\cdot,t)) \compose \h$. Now dividing both sides of \eqref{eq:Ztbarmain} by $\Zapbar$ we get
\begin{align*}
	\frac{1}{\Zapbar}\partial_{t}\Zbar + b &= \frac{\sigma \pap \Theta}{\Zapabs^{2}}
\end{align*}
We now assume that both $\frac{1}{\Zapbar}\partial_{t}\Zbar$ and $\frac{\sigma \pap \Theta}{\Zapabs^{2}}$ go to $0$ as $\abs{\ap} \to \infty$. Taking the imaginary part on both sides we get
\begin{align*}
\Imag\brac{\frac{1}{\Zapbar}\partial_{t}\Zbar} = \Imag\brac{\frac{\sigma \pap \Theta}{\Zapabs^{2}}}
\end{align*}
As $\frac{1}{\Zapbar}\partial_{t}\Zbar$ is the boundary value of an anti-holomorphic function, we have that $\Hil\brac{\frac{1}{\Zapbar}\partial_{t}\Zbar} = - \frac{1}{\Zapbar}\partial_{t}\Zbar$ and hence $i\Hil \brac{\Imag\brac{\frac{1}{\Zapbar}\partial_{t}\Zbar}} = -\Real\brac{\frac{1}{\Zapbar}\partial_{t}\Zbar}$. Therefore applying the operator $i(\Id - \Hil)$ on both sides we get 
\begin{align}\label{eq:dtZbar}
	\frac{1}{\Zapbar}\partial_{t}\Zbar &= i(I-\Hil)\left\{\Imag\left(\frac{\sigma\pap\Th}{\Zapabs^2}\right)\right\}.
\end{align}


\section{Main results}\label{sec:mainresults}

In this section, we will state the main result of this paper. To do so, we first define the notion of self similar solutions to the interface equation \eqref{eq:dtZbar}.

From now on we will fix $\sigma = 1$. Define for $\lamb > 0$
\begin{align}\label{eq:scaling}
Z_{\lambda}(\ap,t) = \lambda^{-a}Z(\lambda\ap,\lambda^b t)
\end{align}
Then we see that
\begin{align*}
(\Z_\lamb)_{\ap}(\ap,t) &= \lambda^{-a+1}\Zap(\lambda\ap,\lambda^{b} t) \\
\Theta_{\lambda}(\ap,t) & = \lambda^{a}\Theta(\lambda\ap,\lambda^{b} t) \\
\partial_{t}(Z_\lamb)(\ap,t) &= \lambda^{-a+b}\partial_{t}Z(\lambda\ap,\lambda^{b} t)
\end{align*}
Now assume that $\Z(\ap,t)$ is a solution to \eqref{eq:dtZbar} with $\sigma = 1$. Then $\Z_\lamb(\ap,t)$ is also a solution if
\begin{align*}
\lambda^{b-1}	\frac{1}{\Zapbar(\lambda\ap,\lambda^{b} t)}\partial_{t}\Zbar(\lambda\ap,\lambda^{b} t) = \lambda^{3a-1} i(I-\Hil)\left\{\Imag\left(\frac{\sigma\pap\Th}{\Zapabs^2}\right)\right\}(\lambda\ap,\lambda^{b} t)
\end{align*}
which gives us $b = 3a$. Hence our main equation \eqref{eq:Ztbarmain} is invariant under the scaling transformation \eqref{eq:scaling} with $b = 3a$. Note that there are no restrictions right now on the value of $a$. We also easily observe that the equation  \eqref{eq:Ztbarmain} is invariant under the transformation
\begin{align}\label{eq:timescaling}
Z(\ap,t) \to \beta Z(\ap,\beta^{-3}t)
\end{align}
for any $\beta >0$.

We now assume that $Z$ is a self  similar solution i.e. $Z_\lamb(\ap,t) = Z(\ap,t)$ for all $\lamb > 0$. This gives us
\begin{align*}
Z(\ap,t) = \lambda^{-a}Z(\lambda\ap,\lambda^b t)
\end{align*}
By choosing
\begin{equation*}
\lamb^bt = \lambda^{3a}t  = 1.
\end{equation*}
we get $\lambda = t^{-\frac{1}{3a}}$. Therefore for $t>0$ we have
\begin{align*}
Z(\ap,t) = t^{\frac{1}{3}}Z(t^{-\frac{1}{3a}}\ap,1)
\end{align*}

At $t=1$, we now assume that $Z(\ap,1)$ is the graph of a smooth Lipschitz function. Further, asymptotically, as $\ap\rightarrow \pm \infty$, we assume at $t=1$ that the angle of the unit tangent vector $\tau$ of the interface with respect to the positive $x -axis$ tends to $\pm\ep$  respectively for some fixed $\ep\in (-\pi/2,\pi/2)$. 

From the above formula, this implies for all $t>0$ that $Z(\cdot,t)$ is the graph of a Lipschitz function. To fix the scaling factor in the Riemann mapping, we impose (recall that the branch cut of log is the positive imaginary axis)
\begin{align}\label{eq:Zasymptoticst1}
	Z(\ap,1) &\simeq e^{i\ep} \ap^{\frac{\pi+2\ep}{\pi}} \qquad \text{as} \qquad \abs{\ap}\rightarrow \infty
\end{align}
This implies that for $t>0$ we have
\begin{align}\label{eq:Zasymptotics}
	Z(\ap,t) &\simeq e^{i\ep}t^{\frac{1}{3}}(t^{-\frac{1}{3a}}\ap)^{\frac{\pi+2\ep}{\pi}} \qquad \text{as}  \qquad \abs{\ap}\rightarrow \infty
\end{align}
With this choice we see that the only ambiguity in the choice of Riemann map is the translation freedom at $t = 1$, which does not affect the analysis in this paper in any way. 

Now observe that to obtain a finite nontrivial limit for $Z(\ap,t)$ as $t \to 0$, we need
\begin{equation*}
\frac{1}{3} - \frac{1}{3a}\frac{\pi+2\ep}{\pi} = 0
\end{equation*}
which implies that
\begin{equation}\label{eq:a}
	a = \frac{\pi+2\ep}{\pi} = 1 +\frac{2\ep}{\pi} > 0 .
\end{equation}
We will call 
\begin{equation*}
	\eta(\ap) = Z(\ap,1)
\end{equation*}
and so
\begin{align}\label{eq:etadefinition}
	Z(\ap,t) = t^{\frac{1}{3}}\eta(t^{-\frac{1}{3a}}\ap)
\end{align}
for the value of $a$ in \eqref{eq:a}.

We can now state our main theorem:
\begin{thm}\label{thm:mainselfsimilarresult}
There exists $\ep_0>0$ such that if  $\ep \in [-\ep_0, \ep_0]$, then there exists a self-similar solution $Z: \Rsp \times [0,\infty) \to \Csp$ to \eqref{eq:dtZbar} such that
\begin{enumerate}
\item For $t>0$, $Z(\ap,t)$ is a smooth function with asymptotics given by \eqref{eq:Zasymptotics}.
\item At $t=0$ we have
\begin{align*}
Z(\ap,0) = 
\begin{cases}
e^{i\ep}\abs{\ap}^{1 + \frac{2\ep}{\pi}} \qq & \tx{ for } \ap \geq 0 \\
e^{-i(\pi + \ep)}\abs{\ap}^{1 + \frac{2\ep}{\pi}} \qq & \tx{ for } \ap < 0 
\end{cases}
\end{align*}
and $\lim_{t \to 0^+} Z(\ap, t) = Z(\ap,0)$ for all $\ap \in \Rsp$. 
\end{enumerate}
\end{thm}

The self-similar solution of the above theorem gives rise to a solution of the Hele-Shaw equation with surface tension given by \eqref{eq:HeleShaw} and \eqref{eq:surfacetension}, as we will see from the argument given at the end of Section \ref{sec:fullequation}. The rest of the paper is devoted to the proof of Theorem \ref{thm:mainselfsimilarresult}.

\section{Equation for Self-Similarity}\label{sec:selfsimilareq}

We want to derive the equation for $\eta$ by plugging \eqref{eq:etadefinition} into \eqref{eq:dtZbar}. We will define the variable $x = t^{-\frac{1}{3a}}\ap$:
\begin{align*}
	Z(\ap,t) = t^{\frac{1}{3}}\eta(x), \quad \partial_{t}x = -t^{-1}  \frac{x}{3a}, \quad \partial_{\ap}x=t^{-\frac{1}{3a}}
\end{align*}
We compute the following relevant terms:
\begin{enumerate}[label=(\emph{\alph*})]
\item \begin{align*}
			\partial_{t}\Zbar(\ap,t) &= \partial_{t}\left(t^{\frac{1}{3}}\bar{\eta}(x)\right)\\
			&=  \frac{1}{3}t^{-\frac{2}{3}}\left( \bar{\eta}(x)- \frac{x}{a}\bar{\eta}_{x}(x)\right)
	\end{align*}

\item \begin{align*}
	\Zap(\ap,t) &= \pap\left(t^{\frac{1}{3}}{\eta}(x)\right)\\
	&= t^{\frac{1}{3}}t^{-\frac{1}{3a}}\eta_{x}(x)
	\end{align*}
	
\item
	
	\begin{align*}
			\frac{1}{\Zapbar(\ap,t)}\partial_{t}\Zbar(\ap,t) &=  \frac{1}{3}t^{-1}t^{\frac{1}{3a}}\frac{1}{\bar{\eta}_{x}(x)}\left( \bar{\eta}(x) - \frac{x}{a}\bar{\eta}_{x}(x)\right)
	\end{align*}
	
\item \begin{align*}
	\Dap\frac{\Zap}{\Zapabs}(\ap,t) &= \frac{1}{ t^{\frac{1}{3}}t^{-\frac{1}{3a}}\eta_{x}(x)}\pap \left(\frac{\eta_{x}(x)}{\abs{\eta_{x}(x)}}\right)\\
	&= t^{-\frac{1}{3}} \frac{1}{\eta_{x}(x)}\frac{d}{dx} \left(\frac{\eta_{x}(x)}{\abs{\eta_{x}(x)}}\right)
	\end{align*}
	
\item \begin{align*}
		\pap\Th(\ap,t) &= \pap\left( -i(\Id + \Hil)\cbrac{ t^{-\frac{1}{3}} \frac{1}{\eta_{x}(x)}\frac{d}{dx} \left(\frac{\eta_{x}(x)}{\abs{\eta_{x}(x)}}\right)}\right)\\
		&= -it^{-\frac{1}{3}}t^{-\frac{1}{3a}}\frac{d}{dx}(\Id + \Hil)\cbrac{ \frac{1}{\eta_{x}(x)}\frac{d}{dx} \left(\frac{\eta_{x}(x)}{\abs{\eta_{x}(x)}}\right)}
	\end{align*}
	
\item \begin{align*}
		\frac{\pap\Th}{\Zapabs^2}(\ap,t) &= -it^{-1}t^{\frac{1}{3a}}\frac{1}{\abs{\eta_{x}(x)}^{2}}\frac{d}{dx}(\Id + \Hil)\cbrac{ \frac{1}{\eta_{x}(x)}\frac{d}{dx} \left(\frac{\eta_{x}(x)}{\abs{\eta_{x}(x)}}\right)}
	\end{align*}
\end{enumerate}

From the above calculations, we obtain the following formula from \eqref{eq:dtZbar}:

\begin{align}\label{eq:etaselfsimilar}
\frac{1}{3}\frac{1}{\bar{\eta}_{x}}\left( \bar{\eta}- \frac{x}{a}\bar{\eta}_{x}\right) &= i(I-\Hil)\cbrac{\Imag\left(-i\frac{1}{\abs{\eta_{x}}^{2}}\frac{d}{dx}(\Id + \Hil)\cbrac{ \frac{1}{\eta_{x}}\frac{d}{dx} \left(\frac{\eta_{x}}{\abs{\eta_{x}}}\right)}\right)} \nonumber\\
&=-i(I-\Hil)\cbrac{\Real\left(\frac{1}{\abs{\eta_{x}}^{2}}\frac{d}{dx}(\Id + \Hil)\cbrac{ \frac{1}{\eta_{x}}\frac{d}{dx} \left(\frac{\eta_{x}}{\abs{\eta_{x}}}\right)}\right)}
\end{align}
As $\eta_{x}$ is the boundary value of the derivative of a Riemann map,  we define real-valued functions $f$ and $g$ by the relation 
\begin{align}\label{def:g}
	\eta_{x} = e^{f+ig}
\end{align}
with the condition that $g$ is continuous and tends to $\pm \ep$ as $x \rightarrow \pm \infty$. We have that
\begin{align*}
\eta_{xx}= (f_{x}+ig_{x})\eta_{x}, \quad	\frac{\bar{\eta}_{xx}}{\bar{\eta}_{x}} &= f_{x} - i g_{x}.  
\end{align*}
Using the above, we will derive an equation in terms of $f$ and $g$ from \eqref{eq:etaselfsimilar}.

To do so, we first apply the following sequence of operations to both sides of \eqref{eq:etaselfsimilar}: first multiply by $\bar{\eta}_x$, then differentiate, then divide by $\bar{\eta}_{x}$. The left hand side of the equation becomes
\begin{align*}
	\frac{1}{3}\frac{1}{\bar{\eta}_{x}}\frac{d}{dx}\brac{ \bar{\eta}- \frac{x}{a}\bar{\eta}_{x}} &= \frac{1}{3}-\frac{1}{3a} - \frac{x}{3a}\frac{\bar{\eta}_{xx}}{\bar{\eta}_{x}}\\
	&= \frac{1}{3}-\frac{1}{3a} - \frac{x}{3a}f_{x} + i\frac{x}{3a}g_{x}.
\end{align*}
Computing similarly on the right hand side, we obtain that
\begin{align}\label{eq:18'}
\begin{aligned}
 & \frac{1}{3}-\frac{1}{3a} - \frac{x}{3a}f_{x} + i\frac{x}{3a}g_{x} \\
 &=\frac{-i}{\bar{\eta}_{x}}\frac{d}{dx}\brac{\bar{\eta}_{x}(I-\Hil)\cbrac{\Real\left(\frac{1}{\abs{\eta_{x}}^{2}}\frac{d}{dx}(\Id + \Hil)\cbrac{ \frac{1}{\eta_{x}}\frac{d}{dx} \left(\frac{\eta_{x}}{\abs{\eta_{x}}}\right)}\right)}}.
\end{aligned}
\end{align}

Finally take the imaginary part of the resulting term. The left hand side of \eqref{eq:18'} becomes
\begin{align*}
\Imag\sqbrac{ \frac{1}{3}-\frac{1}{3a} - \frac{x}{3a}f_{x} + i\frac{x}{3a}g_{x}}
&= \frac{x}{3a}g_{x}
\end{align*}
The right hand side of \eqref{eq:18'} becomes
\begingroup
\allowdisplaybreaks
\begin{align*}
&\Imag\sqbrac{\frac{-i}{\bar{\eta}_{x}}\frac{d}{dx}\brac{\bar{\eta}_{x}(I-\Hil)\cbrac{\Real\left(\frac{1}{\abs{\eta_{x}}^{2}}\frac{d}{dx}(\Id + \Hil)\cbrac{ \frac{1}{\eta_{x}}\frac{d}{dx} \left(\frac{\eta_{x}}{\abs{\eta_{x}}}\right)}\right)}}}\\
&= \Imag\sqbrac{-i\frac{\bar{\eta}_{xx}}{\bar{\eta}_{x}}(I-\Hil)\cbrac{\Real\left(\frac{1}{\abs{\eta_{x}}^{2}}\frac{d}{dx}(\Id + \Hil)\cbrac{ \frac{1}{\eta_{x}}\frac{d}{dx} \left(\frac{\eta_{x}}{\abs{\eta_{x}}}\right)}\right)}}\\*
&\quad -\frac{d}{dx}\cbrac{\Real\left(\frac{1}{\abs{\eta_{x}}^{2}}\frac{d}{dx}(\Id + \Hil)\cbrac{ \frac{1}{\eta_{x}}\frac{d}{dx} \left(\frac{\eta_{x}}{\abs{\eta_{x}}}\right)}\right)}\\
&= \Imag\sqbrac{-i(f_{x}-ig_{x})(I-\Hil)\cbrac{\Real\left(e^{-2f}\frac{d}{dx}(\Id + \Hil)\cbrac{ ie^{-f}g_{x}}\right)}} \\*
&\quad -\frac{d}{dx}\cbrac{\Real\left(e^{-2f}\frac{d}{dx}(\Id + \Hil)\cbrac{ie^{-f}g_{x}}\right)}\\
&= \Imag\sqbrac{-i(f_{x}-ig_{x})(I-\Hil)\cbrac{e^{-2f}i \Hil\frac{d}{dx}\cbrac{ e^{-f}g_{x}}}} -\frac{d}{dx}\cbrac{e^{-2f} i\Hil\frac{d}{dx}\cbrac{e^{-f}g_{x}}}\\
&= if_{x}\cbrac{e^{-2f} \Hil\frac{d}{dx} \cbrac{e^{-f}g_{x}}} +g_{x}\Hil\cbrac{e^{-2f}\Hil\frac{d}{dx}\cbrac{ e^{-f}g_{x}}}  -ie^{-2f} \Hil\frac{d^{2}}{dx^{2}}\cbrac{e^{-f}g_{x}}
\end{align*}
\endgroup
Thus,
\begin{align}\label{eq:xgxselfsimilar}
	 \frac{x}{3a}g_{x} &=  if_{x}\cbrac{e^{-2f} \Hil\frac{d}{dx} \cbrac{e^{-f}g_{x}}} +g_{x}\Hil\cbrac{e^{-2f}\Hil\frac{d}{dx}\cbrac{ e^{-f}g_{x}}}  -ie^{-2f} \Hil\frac{d^{2}}{dx^{2}}\cbrac{e^{-f}g_{x}}
\end{align}
Hence, we have the following equation for $f$ and $g$:
\begin{align}\label{eq:selfsimilar}
i\Hil\frac{d^{2}}{dx^{2}}\cbrac{e^{-f}g_{x}} + e^{2f}\frac{x}{3a}g_{x} &=  if_{x} \Hil\frac{d}{dx} \cbrac{e^{-f}g_{x}} + e^{2f}g_{x}\Hil\cbrac{e^{-2f}\Hil\frac{d}{dx}\cbrac{ e^{-f}g_{x}}} 
\end{align}

\subsection{Linear approximation}

In this subsection, we will begin the analysis of equation \eqref{eq:selfsimilar}. We know that $g\rightarrow \pm \ep$ as $x\rightarrow \pm \infty$. We express $g$ as 
\begin{align}\label{def:gdecompose}
	g = \ep G + u
\end{align}
where $G \to \pm1$ as $x \to \pm\infty$ and  $u \rightarrow 0$ as $x\rightarrow \pm \infty$. As a first step, we find an appropriate choice for $G$. We choose $G$ to be an odd function satisfying the linearized equation from \eqref{eq:selfsimilar}:
\begin{align}\label{eq:linearG}
  i \Hil G_{xxx}+\frac{x}{3a}G_{x} &= 0
\end{align}
Applying the Fourier transform to \eqref{eq:linearG}, we obtain the following equation for $\widehat{G_{x}}(\xi)$:
\begin{align*}
3a\sgn(\xi)\xi^2\widehat{G_{x}} + \frac{d}{d\xi}\widehat{G_{x}} &= 0
\end{align*}
For $\xi > 0$, we obtain the solution
\begin{align*}
\widehat{G_{x}}(\xi) &= C_{+}e^{-a\xi^{3}} = C_{+}e^{-a\abs{\xi}^{3}}.
\end{align*}
for some constant $C_{+} \in \Rsp$. For $\xi < 0$, we obtain the solution
\begin{align*}
	\widehat{G_{x}}(\xi) &= C_{-}e^{a\xi^{3}} = C_{-}e^{-a\abs{\xi}^{3}}.
\end{align*}
for some constant $C_{-} \in \Rsp$. Because  $G \to \pm1$ as $x \to \pm\infty$, observe that
\begin{align*}
	\widehat{G_{x}}(0) &= \frac{1}{\sqrt{2\pi}}\int_{-\infty}^{\infty} G_{x}(x) dx = \frac{\sqrt{2}}{\sqrt{\pi}}.
\end{align*}
Hence, for all $\xi\in\mathbb{R}$, we have the appropriate choice for $G$ as
\begin{align}\label{eq:Gsolution}
		\widehat{G_{x}}(\xi) &= \frac{\sqrt{2}}{\sqrt{\pi}}e^{-a\abs{\xi}^{3}}
\end{align}

\begin{prop}\label{prop:Gproperties}
Let $0 \leq \abs{\ep} \leq 1/2$.  Then the odd function $G$ defined by the relation \eqref{eq:Gsolution} satisfies the following properties:
\begin{enumerate}
\item For any $n \geq 1$ we have $ \abs{\partial_{x}^{n}G(x)} \lesssim_{n} (1+ \abs{x})^{-2-n}$ for all $x \in \Rsp$. In particular we have that $G_x \in H^\infty(\Rsp)$. 
\item For any $n \geq 1$ we have $\abs{\partial_{x}^{n}\Hil G(x)} \lesssim_{n} (1+ \abs{x})^{-n}$ for all $x \in \Rsp$. 
\item There exists $C_1 \in \Rsp$ and $V \in H^{\infty}(\mathbb{R})$ such that for all $x \in \Rsp$ we have
\begin{align*}
\int_{0}^{x} i\Hil G_{x}(y)\diff y = \frac{1}{\pi}\log(a^{2}+x^{2}) + C_1 + V(x)
\end{align*}
\end{enumerate}
\end{prop}
\begin{proof}
First observe from \eqref{eq:a} and $0 \leq \abs{\ep} \leq 1/2$ that we have $1/2 \leq a \leq 3/2$. We now prove the claims sequentially:
\begin{enumerate}
\item For $n \geq 1$ we have that
\begin{align*}
	\widehat{\partial_{x}^{n}G}(\xi) &= \brac{i\xi}^{n-1}\widehat{G_{x}}(\xi) = \brac{i\xi}^{n-1}\frac{\sqrt{2}}{\sqrt{\pi}}e^{-a\abs{\xi}^{3}}.
\end{align*}
Now observe that for any $0 \leq k \leq n+2$ we have
\begin{align*}
\norm[1]{\frac{d^{k}}{d\xi^{k}}\brac{\brac{i\xi}^{n-1}\frac{\sqrt{2}}{\sqrt{\pi}}e^{-a\abs{\xi}^{3}}}} \lesssim_{n} 1
\end{align*}
This implies that
\begin{align*}
\norm[\infty]{(1+ \abs{x})^{2 + n} \partial_{x}^{n}G(x)} \lesssim_n 1
\end{align*}
This proves the first property. 
\item For $n \geq 1$ we have
\begin{align*}
	\widehat{i\Hil \partial_{x}^{n}G}(\xi) &= -i^{n}\xi^{n-1}\frac{\sqrt{2}}{\sqrt{\pi}}\sgn(\xi)e^{-a\abs{\xi}^{3}}\\
	&= -i^{n}\xi^{n-1}\frac{\sqrt{2}}{\sqrt{\pi}}\sgn(\xi)e^{-a\abs{\xi}} -i^{n}\xi^{n-1}\frac{\sqrt{2}}{\sqrt{\pi}}\sgn(\xi)\brac{e^{-a\abs{\xi}^{3}}-e^{-a\abs{\xi}}}.
\end{align*}
Let the inverse Fourier transform of the second term on the right hand side of the above equation be $v_n(x)$. Observe that for any $0 \leq k \leq n+1$ we have
\begin{align*}
\norm[1]{\frac{d^{k}}{d\xi^{k}}\cbrac{\xi^{n-1}\sgn(\xi)\brac{e^{-a\abs{\xi}^{3}}-e^{-a\abs{\xi}}}}} \lesssim_n 1
\end{align*}
and hence we see that
\begin{align*}
\norm[\infty]{(1+\abs{x})^{n+1} v_n(x) } \lesssim_n 1
\end{align*}
Since
\begin{align*}
\frac{\sqrt{2}}{\sqrt{\pi}}\widehat{\frac{x}{a^{2}+x^{2}}} &= -i\sgn(\xi)e^{-a\abs{\xi}},
\end{align*}
we have that
\begin{align*}
-i^{n}\xi^{n-1}\frac{\sqrt{2}}{\sqrt{\pi}}\sgn(\xi)e^{-a\abs{\xi}} &= \Fcal\brac{\partial_{x}^{n-1}\brac{\frac{2}{\pi}\frac{x}{a^{2}+x^{2}}}}.
\end{align*}
The above inequalities directly yields the required property. 
\item From the above calculation we observe that
\begin{align*}
\int_{0}^{x} i\Hil G_{x}(y)\diff y &=\int_{0}^{x} \brac{\frac{2}{\pi}\frac{y}{a^{2}+y^{2}} + v_{1}(y)} \diff y \\
&= \frac{1}{\pi}\log(a^{2} + x^{2}) + C_{1} + V(x)
\end{align*}
where 
\begin{align*}
C_1 =- \frac{1}{\pi}\log(a^{2}) + \int_{0}^{\infty} v_{1}(y) \diff y
\end{align*}
and
\begin{align*}
	V(x) = -\int_{x}^{\infty} v_{1}(y) \diff y.
\end{align*}
It follows from the decay properties of $\hat{v}_1$ and $v_1$ that $V \in H^{\infty}$. 
\end{enumerate}
\end{proof}

As $G$ is an $\Linfty$ function with $G \to \pm 1$ as $x \to \pm \infty$, we know that $\Hil G \in BMO$, where $BMO$ is the space of bounded mean oscillations which are functions modulo constants. It would be helpful to identify $\Hil G$ with a function, so using the last property of \propref{prop:Gproperties} we fix the value of constant of $\Hil G(0)$ so that we have
\begin{align}\label{eq:HilG}
i\Hil G (x) =  \frac{1}{\pi}\log(a^{2} + x^{2}) + \frac{1}{\ep}\ln\brac{1 + \frac{2\ep}{\pi}} + V(x)
\end{align}
Note that we have $\ln(\eta_x) = f+ ig$ from \eqref{def:g} and so $f+ig$ is the boundary value of a holomorphic function on the lower half plane. From the asymptotics of $Z(\ap,1)$ in \eqref{eq:Zasymptoticst1} and the formulae \eqref{eq:HilG}, \eqref{def:gdecompose} we see that 
\begin{align*}
\brac{(f+ig) - i\ep(G + \Hil G)}(x) \to 0 \qq \tx{ as } \abs{x} \to \infty
\end{align*}
As $\Hil(\Hil G) = G$ we therefore have $\Hil(f+ig) = f+ig$ and so $f = i\Hil g$ and $g = -i\Hil f$.

We define the weight $w: \Rsp \to (0,\infty)$ as
\begin{align}\label{def:w}
w  = e^{3i\ep\Hil G}
\end{align}
Using this weight, we define the weighted $L^{2}$ norm of a function $v: \Rsp \to \Csp$ as follows
\begin{align*}
	\norm[L^{2}_{w}]{v} &= \brac{\int_{\mathbb{R}} \abs{v(x)}^{2} w(x) \diff x}^{\frac{1}{2}} = \norm[L^{2}]{w^\half  v}.
\end{align*}

We now prove some basic properties of this weight. 

\begin{prop}\label{prop:weightproperties}
There exists $0< \ep_1 \leq 1/100$ such that for all $0 \leq \abs{\ep} \leq \ep_1$ we have the following:
\begin{enumerate}
\item The constant $a$ defined by \eqref{eq:a} satisfies $1/2 \leq a \leq 3/2$.
\item We have $\ep_1\norm[L^{\infty}]{ x\Hil G_{x}} \leq 1/6$.
\item For all $x \in \Rsp$ we have
\begin{align*}
\brac{1+ \abs{x}}^{\frac{6\ep}{\pi}} \lesssim w(x)\lesssim \brac{1+ \abs{x}}^{\frac{6\ep}{\pi}}
\end{align*}
In particular we have for all $x \in \Rsp$
\begin{align*}
w(x), \frac{1}{w(x)} &\lesssim \brac{1+ \abs{x}}^{\frac{1}{1000}}
\end{align*}
\item For any $n \geq 1$ we have for all $x \in \Rsp$
\begin{align*}
\abs{\partial_{x}^{n}w(x)} &\lesssim_{n} \abs{\ep} \brac{1+ \abs{x}}^{-n + \frac{1}{1000}}.
\end{align*}
\item For all $x,y \in \Rsp$ with $x \neq y$ we have
\begin{align*}
\abs{\frac{w(x)-w(y)}{x-y}} \lesssim 1, \quad \abs{\frac{xw(x)-yw(y)}{x-y}} \lesssim 1+\abs{x}^{\frac{1}{1000}}+\abs{y}^{\frac{1}{1000}}.
\end{align*}
\item If $v: \Rsp \to \Csp$ is a function, then we have the equivalence
\begin{align*}
\norm[L^{2}_{w}]{v}^{2} & \lesssim  \int_{\mathbb{R}} \abs{v}^{2}(x w)_{x} \diff x \lesssim \norm[L^{2}_{w}]{v}^{2}.
\end{align*}
\end{enumerate}
\end{prop}
\begin{proof}
We prove these sequentially.
\begin{enumerate}
\item This follows directly from \eqref{eq:a}.
\item This follows directly from \propref{prop:Gproperties}.
\item This follows from \eqref{eq:HilG} and choosing $\ep_0>0$ small enough. 
\item Using \eqref{eq:HilG} and \propref{prop:Gproperties} we get the required inequality. 
\item This follows from the above inequality directly.
\item To see this, we observe that
\begin{align*}
\int_\Rsp \abs{v}^2(xw)_x \diff x = \int_\Rsp \abs{v}^2 w \brac{1 + x \frac{w_x}{w}} \diff x = \int_\Rsp \abs{v}^2 w \brac{1 + 3i\ep x\Hil G_x} \diff x
\end{align*}
Now using the fact that $\ep_1\norm[L^{\infty}]{ x\Hil G_{x}} \leq 1/6$, we get the required equivalence. 
\end{enumerate}
\end{proof}

From this point forward in the paper, we will always assume that $\ep \in [-\ep_1, \ep_1]$ where $\ep_1>0$ is a number  given by \propref{prop:weightproperties}.

\subsection{Derivation of the main equation}

Using $f = \ep i \Hil G + i\Hil u$, we observe that
\begin{align*}
e^{3f} = e^{3i\ep\Hil G} e^{3i\Hil u} = we^{3i\Hil u} = w + w(e^{3i\Hil u} - 1)
\end{align*}
Now
\begin{align*}
& i\Hil\frac{d^{2}}{dx^{2}}\cbrac{e^{-f}g_{x}} \\
& = i\Hil\frac{\diff}{\diff x} \cbrac{-e^{-f}f_xg_x + e^{-f}g_{xx}} \\
& = i\Hil\cbrac{e^{-f}f_x^2 g_x - e^{-f}f_{xx}g_x -2e^{-f}f_xg_{xx} + e^{-f}g_{xxx}} \\
& = e^{-f}i\Hil g_{xxx} - i\sqbrac{e^{-f}, \Hil}g_{xxx} + i\Hil\cbrac{e^{-f}f_x^2 g_x - e^{-f}f_{xx}g_x -2e^{-f}f_xg_{xx}}
\end{align*}
Multiplying both sides of the above equation by $e^f$ and using $g = \ep G + u$ we get
\begin{align*}
& i\Hil u_{xxx} \\
& = - \ep i\Hil G_{xxx} + i\Hil g_{xxx}  \\
& = -\ep i\Hil G_{xxx} + e^f i\Hil\frac{d^{2}}{dx^{2}}\cbrac{e^{-f}g_{x}} + ie^f \sqbrac{e^{-f}, \Hil}g_{xxx} \\
& \quad - e^f i\Hil\cbrac{e^{-f}f_x^2 g_x - e^{-f}f_{xx}g_x -2e^{-f}f_xg_{xx}}
\end{align*}
We also have
\begin{align*}
w\frac{x}{3a} u_x & = - \ep w\frac{x}{3a}G_x + w\frac{x}{3a}g_x \\
& = - \ep w\frac{x}{3a}G_x + e^{3f}\frac{x}{3a}g_x - w(e^{3i\Hil u} - 1)\frac{x}{3a}g_x
\end{align*}
Adding these two equations and using \eqref{eq:selfsimilar} we therefore get the main equation
\begin{align}\label{eq:linearequationorig}
\Lcal u = F
\end{align}
where
\begin{align}\label{def:L}
\Lcal u = i\Hil u_{xxx} + w\frac{x}{3a} u_x
\end{align}
and
\begin{align}\label{F}
\begin{aligned}
F & =  -\ep i\Hil G_{xxx} + ie^f \sqbrac{e^{-f}, \Hil}g_{xxx}  - e^f i\Hil\cbrac{e^{-f}f_x^2 g_x - e^{-f}f_{xx}g_x -2e^{-f}f_xg_{xx}} \\
& \quad - \ep w\frac{x}{3a}G_x - w(e^{3i\Hil u} - 1)\frac{x}{3a}g_x + ie^f f_{x} \Hil\frac{d}{dx} \cbrac{e^{-f}g_{x}} \\
& \quad + e^{3f}g_{x}\Hil\cbrac{e^{-2f}\Hil\frac{d}{dx}\cbrac{ e^{-f}g_{x}}}
\end{aligned}
\end{align}

\section{Linear Analysis}\label{sec:linearanalysis}

In this section we analyze the linear equation
\begin{align}\label{eq:linearequation}
\Lcal v = F
\end{align}
Here $\Lcal$ is defined in \eqref{def:L} namely
\begin{align*}
\Lcal v = i\Hil v_{xxx} + w\frac{x}{3a} v_x
\end{align*}
and $w$ is defined by \eqref{def:w}. We define the norm
\begin{align}\label{def:Xnorm}
	\norm[X]{v} &= \norm[L^{2}_{w}]{v} + \norm[L^{2}_{w}]{xv_{x}} + \norm[\dot{H}^{1}]{v} +  \norm[\dot{H}^{3}]{v}.
\end{align}

The remainder of this section is dedicated to proving the following theorem.

\begin{thm}\label{thm:linear}
Let $\ep_1>0$ be given by \cref{prop:weightproperties} and let $\ep \in [-\ep_1, \ep_1]$. If $F\in L^{2}(\mathbb{R})$ and $w^{-\frac{1}{2}}F\in L^{2}(\mathbb{R})$, then exists a unique solution $v \in L^{2}_{w}(\mathbb{R}) \cap \dot{H}^{\frac{3}{2}}(\Rsp)$ to \eqref{eq:linearequation} in the distributional sense. Moreover $v \in \dot{H}^{1}(\Rsp)\cap\dot{H}^{3}(\Rsp)$ and $xv_{x}\in L^{2}_{w}(\mathbb{R})$ with
\begin{align}\label{ineq:linearXestimate}
\norm[X]{v} &\lesssim \norm[L^{2}]{F} + \norm[L^{2}]{w^{-\frac{1}{2}}F}.
\end{align}
Furthermore, $v$ satisfies the additional estimate
\begin{align*}
\norm[L^{2}]{(1+x^{2})^{\frac{1}{8}}v_{xx}} + \norm[L^{\infty}]{(1+x^{2})^{\frac{1}{8}}v} &\lesssim  \norm[X]{v} \lesssim \norm[L^{2}]{F} + \norm[L^{2}]{w^{-\frac{1}{2}}F}. 
\end{align*}
\end{thm}

The proof is structured as follows. In subsection \ref{subsec:weakexist}, we demonstrate the existence of a weak solution to \eqref{eq:linearequation}. Next, in subsection \ref{subsec:decaygain}, we prove \eqref{ineq:linearXestimate}. Uniqueness of the solution is proven in subsection \ref{subsec:uniq}. The second inequality in Theorem \ref{thm:linear} is proven in subsection \ref{subsec:furtherproperties}.

\subsection{Existence of a Weak Solution to the Linear Equation}\label{subsec:weakexist}

Define the bilinear map $B: L^{2}_{w}(\mathbb{R})\cap \dot{H}^{\frac{3}{2}}(\mathbb{R}) \times \mathcal{D}(\mathbb{R})\rightarrow \mathbb{R}$ as
\begin{align}\label{def:B}
	B(v,\zeta) = \int_{\mathbb{R}} v \abs{D}^{3}\zeta + v\brac{w\frac{x}{3a}\zeta}_{x} \diff x.
\end{align}
Let $w^{-\frac{1}{2}}F \in L^{2}(\mathbb{R})$. We say that $v \in L^{2}_{w}(\mathbb{R})\cap \dot{H}^{\frac{3}{2}}(\mathbb{R})$ solves \eqref{eq:linearequation} in weak sense if 
\begin{align}\label{eq:weakselfsimilar}
	B(v,\zeta) = -\int_\Rsp F\zeta \diff x \qquad\quad \text{for all }  \zeta\in \mathcal{D}(\mathbb{R})
\end{align}

Note that $v$ solves \eqref{eq:linearequation} in weak sense is the same as $v$ solving \eqref{eq:linearequation} in the sense of distributions. To prove the existence of a solution to \eqref{eq:weakselfsimilar}, we will make use of the classical Lax-Milgram theorem \cite{La02}:
\begin{theorem}[Lax-Milgram]
	Let $\mathcal{H}$ be a real Hilbert space and suppose $A: \mathcal{H}\times\mathcal{H} \rightarrow \mathbb{R}$ is a bilinear form satisfying the following properties:
\begin{enumerate}
\item It is bounded,  i.e. there exists $C>0$ such that
\begin{align*}
\abs{A(v,\zeta)} \leq C\norm[\mathcal{H}]{v}\norm[\mathcal{H}]{\zeta} \quad \text{for any  } v,\zeta \in \mathcal{H}
\end{align*}
\item It is coercive, i.e. there exists $c>0$ such that
\begin{align*}
\abs{A(v,v)} \geq c\norm[\mathcal{H}]{v}^{2}	 \quad \text{for any  } v \in \mathcal{H}
\end{align*}
\end{enumerate}
Then, for any $F \in \mathcal{H}^{*}$, there exists a unique  $v\in\mathcal{H}$ satisfying
\begin{align*}
A(v,\zeta) = \abrac{F,\zeta}  \quad \text{for all  } \zeta \in \mathcal{H}.
\end{align*}
Moreover this $v$ satisfies the following estimate
\begin{align*}
\norm[\mathcal{H}]{v} \leq \frac{1}{c}\norm[\mathcal{H}^{*}]{F}
\end{align*}
\end{theorem}

To apply the Lax-Millgram theorem, we first define an appropriate space $\mathcal{H}$ motivated by the bilinear form \eqref{def:B}. Define the norm
\begin{align*}
	\norm[\mathcal{H}]{v} &=	\norm[L^{2}_{w}]{v} + 	\norm[L^{2}_{w}]{xv_{x}} + 	\norm[\dot{H}^{\frac{3}{2}}]{v}
\end{align*}
For $0 \leq \abs{\ep} \leq \ep_1$, we define the real Hilbert space $\mathcal{H}$ as the closure of compactly supported smooth functions under the non-degenerate norm $\norm[\mathcal{H}]{\cdot}$.

\begin{lemma}\label{lem:Hcalasymp}
Let $v \in \Hcal$. Then we have $xw v^2 \in \Linfty(\Rsp)$ with
\begin{align*}
\norm[\infty]{xwv^2} \lesssim \norm[L^2_w]{v}^2 + \norm[L^2_w]{xv_x}^2
\end{align*}
Furthermore $\abs{xwv^2} \to 0$ as $\abs{x} \to \infty$. 
\end{lemma}
\begin{proof}
Let $v \in \Hcal$ and observe that
\begin{align*}
(xw v^2)_x = (xw)_xv^2 + 2wxvv_x
\end{align*}
Now integrating, using \propref{prop:weightproperties} and using the fact that $xwv^2$ is zero at $x = 0$, we get the estimate 
\begin{align*}
\norm[\infty]{xwv^2} \lesssim \norm[L^2_w]{v}^2 + \norm[L^2_w]{xv_x}^2 \implies \norm[\infty]{\abs{x}^\half w^\half v} \lesssim \norm[L^2_w]{v} + \norm[L^2_w]{xv_x}
\end{align*}
Let $v_n \in \Dcal$ be a sequence of smooth functions with compact support such that $\norm[\Hcal]{v - v_n} \to 0$ as $n \to 0$.  From the above estimate we see that $\norm[\infty]{\abs{x}^\half w^\half (v - v_n)} \to 0$ as $n \to \infty$. Hence $\abs{xwv^2} \to 0$ as $\abs{x} \to \infty$. 
\end{proof}

The bilinear form \eqref{def:B} does not satisfy the coercivity hypothesis of the Lax-Milgram theorem for the Hilbert space $\mathcal{H}$. To resolve this, for $0 < \delta \leq 1$ we define the modified bilinear form $B_{\delta} : \mathcal{H}\times \mathcal{H}\rightarrow \mathbb{R}$ as
\begin{align}\label{def:bilinearformdelta}
	B_{\delta}(v,\zeta) &=  \int_{\mathbb{R}} v \abs{D}^{3}\zeta + v\brac{w\frac{x}{3a}\zeta}_{x} + \delta w x^{2}v_{x} \zeta_{x}\diff x
\end{align}
For every $0 < \delta \leq 1$, we see that the bilinear form $B_{\delta}$ is bounded for $v, \zeta \in \mathcal{H}$ by the Cauchy-Schwarz inequality.
Using \propref{prop:weightproperties} and \lemref{lem:Hcalasymp}, we see that there exists a universal constant $c>0$ such that for all $v \in \Hcal$ we have
\begin{align*}
B_{\delta}(v,v) &\geq \norm[\dot{H}^{\frac{3}{2}}]{v}^{2} + c\norm[L^{2}_{w}]{v}^{2} + \delta \norm[L^{2}_{w}]{xv_{x}}^{2} 
\end{align*}
Observe that if $w^{-\frac{1}{2}}F \in L^{2}$, then $F\in \mathcal{H}^{*}$. Hence by the Lax-Milgram theorem we now know that there exists a unique  $v^{(\delta)} \in \Hcal$ such that 
\begin{align*}
B_\delta(v,\zeta) = -\int_\Rsp F\zeta \diff x \qquad\quad \text{for all }  \zeta\in \Hcal
\end{align*}
For each $0 < \delta \leq 1$, we also see that
 \begin{align*}
 	\norm[\dot{H}^{\frac{3}{2}}]{v^{(\delta)}}^{2} + c\norm[L^{2}_{w}]{v^{(\delta)}}^{2} + \delta \norm[L^{2}_{w}]{xv^{(\delta)}_{x}}^{2}  &\leq B_\delta(v^{(\delta)},v^{(\delta)}) \lesssim \norm[L^{2}]{w^{-\frac{1}{2}}F} \norm[L^{2}_{w}]{v^{(\delta)}}.
 \end{align*}
Thus, we obtain that
 \begin{align*}
 	\norm[\dot{H}^{\frac{3}{2}}]{v^{(\delta)}}^{2} + \frac{c}{2}\norm[L^{2}_{w}]{v^{(\delta)}}^{2} +  \delta \norm[L^{2}_{w}]{xv^{(\delta)}_{x}}^{2}  &\lesssim  \norm[L^{2}]{w^{-\frac{1}{2}}F}^{2}.
 \end{align*}
Hence, if $w^{-\frac{1}{2}}F\in L^{2}(\mathbb{R})$, then the set of functions $v^{(\delta)}$ is uniformly bounded in $\dot{H}^{\frac{3}{2}}(\mathbb{R}) \cap L^{2}_{w}(\mathbb{R})$ as $\delta \to 0$. Thus, there exists a subsequence that converges weakly to some $v \in \dot{H}^{\frac{3}{2}}(\mathbb{R}) \cap L^{2}_{w}(\mathbb{R})$. It is also easy to see that this $v$ is a solution to \eqref{eq:weakselfsimilar}. This proves the existence part of \thmref{thm:linear}. 

We have the following additional regularity by interpolation:
\begin{prop}\label{prop:regularityvinterpolation}
If $v\in  L^{2}_{w}(\mathbb{R}) \cap \dot{H}^{\frac{3}{2}}(\mathbb{R})$, then $w^{\frac{1}{6}}v_{x} \in L^{2}(\mathbb{R})$ and 
\begin{align*}
\norm[L^{2}]{w^{\frac{1}{3}}\abs{D}^{\frac{1}{2}}v} + 	\norm[L^{2}]{w^{\frac{1}{6}}v_{x}} &\lesssim \norm[L^{2}_{w}]{v} + \norm[\dot{H}^{\frac{3}{2}}]{v}.
\end{align*}
\end{prop}
\begin{proof}
Set $W = w^{\frac{1}{6}}$. By \propref{prop:weightproperties}, $W = w^{\frac{1}{6}}$ satisfies the sufficient conditions of the hypothesis in Proposition \ref{prop:weightsobolevinterpolation} and \propref{prop:JdeltaHilnorm}. Moreover, for $0 \leq \abs{\ep} \leq \ep_1$, the constant $C_{W}$ in  Proposition \ref{prop:weightsobolevinterpolation} is bounded by some universal constant. Thus we have the required estimates. 
\end{proof}

\subsection{Decay and Gain of Regularity of the Linear Solution}\label{subsec:decaygain}

If $v\in L^{2}_{w}(\mathbb{R})\cap \dot{H}^{\frac{3}{2}}(\mathbb{R})$ solves \eqref{eq:weakselfsimilar}, then it solves \eqref{eq:linearequation} in a distributional sense:
\begin{align*}
	\mathcal{L}v =  i\Hil v_{xxx} + w\frac{x}{3a}v_{x} = F.
\end{align*}
We will derive a mollified equation that holds pointwise. Let $\phi:\mathbb{R}\rightarrow [0,\infty)$ be a smooth function supported in $(-1,1)$ with $\int_\Rsp \phi(x) \diff x = 1$. For $0 < \delta \leq 1$, define the smoothing operator
\begin{align}\label{def:Jdelta}
J_{\delta} f &= \phi_{\delta}\ast f
\end{align}
where
\begin{align*}
	\phi_{\delta}(x) = \frac{1}{\delta}\phi\brac{\frac{x}{\delta}}.
\end{align*}
Applying $J_{\delta}$ to the linear equation \eqref{eq:linearequation}, we obtain the following equation
\begin{align*}
\brac{i\Hil J_{\delta}v}_{xxx} + J_{\delta}\brac{w\frac{x}{3a}v_{x}} = J_{\delta}F.
\end{align*}
Since $v\in L^{2}_{w}(\mathbb{R})$, we see that $\Hil J_{\delta} v \in L^{2}_{w}(\mathbb{R})$ by Proposition \ref{prop:JdeltaHilnorm} and it is a smooth function. Furthermore, as $v_{x} \in L^{1}_{loc}(\mathbb{R})$, the above equation is true pointwise everywhere. Thus,
\begin{align*}
	\brac{i\Hil J_{\delta}v}_{xxx} + w\frac{x}{3a}\brac{J_{\delta}v}_{x} = J_{\delta}F + \sqbrac{w\frac{x}{3a}, J_{\delta}}v_{x}.
\end{align*}
Next, consider a smooth function $\chi:\mathbb{R}\rightarrow\mathbb{R}$ compactly supported in $[-1,1]$ such that $0\leq \chi(x)\leq 1$ for all $x\in\mathbb{R}$ and $\chi(x) = 1$ if $x\in[-\frac{1}{2},\frac{1}{2}]$. Then, define the smooth cut-off function compactly supported in $[-\delta^{-1},\delta^{-1}]$,
\begin{align}\label{def:chidelta}
\chi_{\delta}(x) = \chi(\delta x).
\end{align}
Multiplying both sides by $\chi_{\delta}$, we obtain that
\begin{align*}
	\chi_{\delta}\brac{i\Hil J_{\delta}v}_{xxx} + \chi_{\delta} w\frac{x}{3a}\brac{J_{\delta}v}_{x} = \chi_{\delta} J_{\delta}F + \chi_{\delta}\sqbrac{w\frac{x}{3a}, J_{\delta}}v_{x}.
\end{align*}
We therefore obtain that
\begin{align}\label{eq:lineardelta}
\mathcal{L}\brac{\chi_{\delta}J_{\delta}v} = F_{\delta} + w\frac{x}{3a}\brac{\partial_{x}\chi_{\delta}}J_{\delta}v
\end{align}
where
\begin{align}\label{def:Fdelta}
	\begin{aligned}
F_{\delta} &= \chi_{\delta} J_{\delta}F + \chi_{\delta}\sqbrac{w\frac{x}{3a}, J_{\delta}}v_{x} + 3\brac{\partial_{x}\chi_{\delta}}\brac{i\Hil J_{\delta}v}_{xx} + 3\brac{\partial_{x}^{2}\chi_{\delta}}\brac{i\Hil J_{\delta}v}_{x} \\
&\qquad  + \brac{\partial_{x}^{3}\chi_{\delta}}\brac{i\Hil J_{\delta}v}- \brac{\sqbrac{\chi_{\delta},i\Hil}J_{\delta}v}_{xxx}
	\end{aligned}
\end{align}
We will use the mollified equation \eqref{eq:lineardelta} in this section to prove regularity and decay of the solution $v$.

Prior to the proof, we note that by Proposition \ref{prop:JdeltaHilnorm}, if $w^{b} f \in L^{2}(\mathbb{R})$ for $0\leq b\leq 1$, then
\begin{align*}
	\norm[L^{2}]{w^{b}\Hil f} + \norm[L^{2}]{w^{b}J_{\delta}f} \lesssim \norm[L^{2}]{w^{b}f}.
\end{align*}
In particular, we will use this fact in the cases  $b= \frac{1}{2}$  and $b=\frac{1}{6}$ with $f= v$ and $f=v_{x}$ respectively. Additionally, note that throughout the remainder of this paper, we will use Proposition \ref{prop:standardcommutator} ubiquitously without reference. 
\begin{prop}\label{prop:chideltaJdeltanorm}
	If $v\in L^{2}_{w}(\mathbb{R}) \cap \dot{H}^{\frac{3}{2}}(\mathbb{R})$, then for all $0<\delta\leq 1$, we have that $\chi_{\delta}J_{\delta}v  \in L^{2}_{w}(\mathbb{R}) \cap \dot{H}^{\frac{3}{2}}(\mathbb{R})$ and
	\begin{align*}
	\norm[L^{2}_{w}]{\chi_{\delta}J_{\delta}v} + \norm[\dot{H}^{\frac{3}{2}}]{\chi_{\delta}J_{\delta}v} &\lesssim \norm[L^{2}_{w}]{v} + \norm[\dot{H}^{\frac{3}{2}}]{v}.
	\end{align*}
Moreover, if additionally we have that $v\in \dot{H}^{1}(\mathbb{R})$, then 
\begin{align*}
\lim_{\delta\rightarrow 0} \brac{\norm[L^{2}_{w}]{\chi_{\delta}J_{\delta}v - v} +	\norm[\dot{H}^{\frac{3}{2}}]{\chi_{\delta}J_{\delta}v - v} } = 0.
\end{align*}
\end{prop}

\begin{proof}
We have that
\begin{align*}
\norm[\dot{H}^{\frac{3}{2}}]{\chi_{\delta}J_{\delta}v} &\leq \norm[\dot{H}^{\frac{1}{2}}]{\brac{\partial_{x}\chi_{\delta}}J_{\delta}v}	+ \norm[\dot{H}^{\frac{1}{2}}]{\chi_{\delta}J_{\delta}v_{x}}.
\end{align*}
First,
\begin{align*}
\norm[\dot{H}^{\frac{1}{2}}]{\brac{\partial_{x}\chi_{\delta}}J_{\delta}v} &\leq \norm[L^{2}]{\brac{\partial_{x}\chi_{\delta}}\abs{D}^{\frac{1}{2}}J_{\delta}v} + \norm[L^{2}]{\sqbrac{\abs{D}^{\frac{1}{2}},\partial_{x}\chi_{\delta}}J_{\delta}v}.
\end{align*}
We use Proposition \ref{prop:regularityvinterpolation} and Proposition \ref{prop:JdeltaHilnorm} to get that
\begin{align*}
\norm[L^{2}]{\brac{\partial_{x}\chi_{\delta}}\abs{D}^{\frac{1}{2}}J_{\delta}v} &\lesssim \delta^{\frac{1}{2}} \norm[L^{2}]{w^{\frac{1}{3}}\abs{D}^{\frac{1}{2}}J_{\delta} v} \lesssim  \delta^{\frac{1}{2}}\brac{ \norm[L^{2}_{w}]{v} + \norm[\dot{H}^{\frac{3}{2}}]{v}}.
\end{align*}
For the other term, note that
\begin{align*}
\sqbrac{\abs{D}^{\frac{1}{2}},\partial_{x}\chi_{\delta}}J_{\delta}v &= \sqbrac{\abs{D}^{\frac{1}{2}},\partial_{x}\chi_{\delta}}\brac{w^{-\frac{1}{2}}w^{\frac{1}{2}}J_{\delta}v}\\
&=\sqbrac{\abs{D}^{\frac{1}{2}},w^{-\frac{1}{2}}\partial_{x}\chi_{\delta}}\brac{w^{\frac{1}{2}}J_{\delta}v }- \partial_{x}\chi_{\delta}\sqbrac{\abs{D}^{\frac{1}{2}},w^{-\frac{1}{2}}}\brac{w^{\frac{1}{2}}J_{\delta}v}
\end{align*}
for which we have the bounds
\begin{align*}
\norm[L^{2}]{\sqbrac{\abs{D}^{\frac{1}{2}},w^{-\frac{1}{2}}\partial_{x}\chi_{\delta}}\brac{w^{\frac{1}{2}}J_{\delta}v }} &\lesssim \norm[L^{2}]{\abs{D}\brac{w^{-\frac{1}{2}}\partial_{x}\chi_{\delta}}}\norm[L^{2}_{w}]{J_{\delta}v}\\
&\lesssim \norm[L^{2}]{\partial_{x}\brac{w^{-\frac{1}{2}}\partial_{x}\chi_{\delta}}}\norm[L^{2}_{w}]{v}\\
&\lesssim \delta \norm[L^{2}_{w}]{v}
\end{align*}
and similarly,
\begin{align*}
\norm[L^{2}]{\partial_{x}\chi_{\delta}\sqbrac{\abs{D}^{\frac{1}{2}},w^{-\frac{1}{2}}}\brac{w^{\frac{1}{2}}J_{\delta}v}} &\lesssim \delta \norm[L^{2}]{\sqbrac{\abs{D}^{\frac{1}{2}},w^{-\frac{1}{2}}}\brac{w^{\frac{1}{2}}J_{\delta}v}}\\
&\lesssim \delta \norm[L^{2}]{\partial_{x}\brac{w^{-\frac{1}{2}}}}\norm[L^{2}_{w}]{J_{\delta}v}\\
&\lesssim \delta \norm[L^{2}_{w}]{v}.
\end{align*}
Second, we look that the term $\norm[\dot{H}^{\frac{1}{2}}]{\chi_{\delta}J_{\delta}v_{x}}$. We have that
\begin{align*}
\norm[\dot{H}^{\frac{1}{2}}]{\chi_{\delta}J_{\delta}v_{x}} &\leq \norm[L^{2}]{\chi_{\delta}J_{\delta}\abs{D}^{\frac{1}{2}}v_{x}} + \norm[L^{2}]{\sqbrac{\abs{D}^{\frac{1}{2}},\chi_{\delta}}J_{\delta}v_{x}}.
\end{align*}
For the latter term, we expand
\begin{align*}
\sqbrac{\abs{D}^{\frac{1}{2}},\chi_{\delta}}J_{\delta}v_{x} &= \sqbrac{\abs{D}^{\frac{1}{2}},w^{-\frac{1}{6}}\chi_{\delta}}\brac{w^{\frac{1}{6}}J_{\delta}v_{x}} - \chi_{\delta}\sqbrac{\abs{D}^{\frac{1}{2}},w^{-\frac{1}{6}}}\brac{w^{\frac{1}{6}}J_{\delta}v_{x}}.
\end{align*}
By similar arguments to above, we obtain that
\begin{align*}
\norm[L^{2}]{\sqbrac{\abs{D}^{\frac{1}{2}},\chi_{\delta}}J_{\delta}v_{x}} &\lesssim \norm[L^{2}]{w^{\frac{1}{6}}v_{x}}.
\end{align*}
Finally, we have that
\begin{align*}
\norm[L^{2}]{\chi_{\delta}J_{\delta}\abs{D}^{\frac{1}{2}}v_{x}} &\leq \norm[\dot{H}^{\frac{3}{2}}]{v}.
\end{align*}
In summary, using Proposition \ref{prop:regularityvinterpolation}, we have shown that
	\begin{align*}
 \norm[\dot{H}^{\frac{3}{2}}]{\chi_{\delta}J_{\delta}v} &\lesssim \norm[L^{2}_{w}]{v} + \norm[\dot{H}^{\frac{3}{2}}]{v}.
\end{align*}
The estimate for $\norm[L^{2}_{w}]{\chi_{\delta}J_{\delta}v} $ follows from Proposition \ref{prop:JdeltaHilnorm}.
Now, let us consider
\begin{align*}
\norm[\dot{H}^{\frac{3}{2}}]{\chi_{\delta}J_{\delta}v - v} &\leq  \norm[\dot{H}^{\frac{3}{2}}]{\chi_{\delta}J_{\delta}v - J_{\delta}v}  + \norm[\dot{H}^{\frac{3}{2}}]{J_{\delta}v - v}.
\end{align*}
We know that 
\begin{align*}
	\lim_{\delta\rightarrow 0} \norm[\dot{H}^{\frac{3}{2}}]{J_{\delta}v - v} = 0.
\end{align*}
For the first term, we have that
\begin{align*}
	\norm[\dot{H}^{\frac{3}{2}}]{\brac{\chi_{\delta}-1}J_{\delta}v } &\leq 	\norm[\dot{H}^{\frac{1}{2}}]{\brac{\chi_{\delta}-1}_{x}J_{\delta}v } + 	\norm[\dot{H}^{\frac{1}{2}}]{\brac{\chi_{\delta}-1}J_{\delta}v_{x} }.
\end{align*}
We have already shown that
\begin{align*}
\norm[\dot{H}^{\frac{1}{2}}]{\brac{\chi_{\delta}-1}_{x}J_{\delta}v }  &= \norm[\dot{H}^{\frac{1}{2}}]{\brac{\partial_{x}\chi_{\delta}}J_{\delta}v } \lesssim \delta^{\frac{1}{2}}\brac{\norm[L^{2}_{w}]{v} + \norm[\dot{H}^{\frac{3}{2}}]{v}}
\end{align*}
Next,
\begin{align*}
\norm[\dot{H}^{\frac{1}{2}}]{\brac{\chi_{\delta}-1}J_{\delta}v_{x} } &\leq \norm[L^{2}]{\brac{\chi_{\delta}-1}\abs{D}^{\frac{1}{2}}J_{\delta}v_{x} }  + \norm[L^{2}]{\sqbrac{\abs{D}^{\frac{1}{2}},\brac{\chi_{\delta}-1}}J_{\delta}v_{x} }
\end{align*}
For the first term, observe that $\abs{D}^{\frac{1}{2}}J_{\delta}v_{x}$ is uniformly dominated by $M\brac{\abs{D}^{\frac{1}{2}}v_{x}}$, where $M$ is the Hardy-Littlewood maximal operator. Therefore, the first term goes to zero as $\delta$ goes to zero by the Dominated Convergence Theorem.

 Since we now additionally assume that $v\in\dot{H}^{1}(\mathbb{R})$, we obtain that
 \begin{align*}
 \norm[L^{2}]{\sqbrac{\abs{D}^{\frac{1}{2}},\brac{\chi_{\delta}-1}}J_{\delta}v_{x} } &\lesssim \norm[L^{2}]{\partial_{x}\chi_{\delta}}\norm[\dot{H}^{1}]{v} \lesssim \delta^{\frac{1}{2}}\norm[\dot{H}^{1}]{v} .
 \end{align*}
Thus,
\begin{align*}
\lim_{\delta\rightarrow 0} 	\norm[\dot{H}^{\frac{3}{2}}]{\chi_{\delta}J_{\delta}v - v} = 0.
\end{align*}
Finally we observe that for all $0 < \delta \leq 1$ we have
\begin{align*}
\abs{\chi_{\delta}J_{\delta}v}(x) \lesssim M(v)(x)
\end{align*}
where $M$ is the maximal operator. Hence from \propref{prop:JdeltaHilnorm} and Dominated Convergence Theorem, we see that $ \lim_{\delta \to 0}\norm[L^{2}_{w}]{\chi_{\delta}J_{\delta}v - v} = 0$.
\end{proof}

We now turn our attention to the term $F_{\delta}$:

\begin{prop}\label{prop:Fdeltaconvergence}
If $v\in  L^{2}_{w}(\mathbb{R})\cap \dot{H}^{\frac{3}{2}}(\mathbb{R})$ and $F, w^{-\frac{1}{2}}F\in   L^{2}(\mathbb{R})$, then for any $0<\delta \leq 1$, we have that $ F_{\delta} , w^{-\frac{1}{2}}F_{\delta}\in L^{2}(\mathbb{R})$, where $ F_{\delta}$ is defined in \eqref{def:Fdelta}. Moreover,
\begin{align*}
\norm[L^{2}]{F_{\delta}} +  \norm[L^{2}]{w^{-\frac{1}{2}} F_{\delta}} &\lesssim \delta^{\frac{1}{4}}\brac{\norm[L^{2}_{w}]{v} + \norm[\dot{H}^{\frac{3}{2}}]{v}} + \norm[L^{2}]{F} +  \norm[L^{2}]{w^{-\frac{1}{2}} F}.
\end{align*}
\end{prop}
\begin{proof}
Throughout the proof, we will assume $b=0$ or $b=-\frac{1}{2}$ in every estimate. First, by Proposition \ref{prop:JdeltaHilnorm}, it follows that
\begin{align*}
 \norm[L^{2}]{\chi_{\delta}J_{\delta}F} +  \norm[L^{2}]{w^{-\frac{1}{2}} \chi_{\delta}J_{\delta} F} &\lesssim \norm[L^{2}]{F} +  \norm[L^{2}]{w^{-\frac{1}{2}} F}.
\end{align*}
For the second term in \eqref{def:Fdelta}, note that by the compact support of $\phi$ in the definition of $J_{\delta}$ and the definition of $\chi$, we get the following equality
\begin{align*}
\chi_{\delta}\sqbrac{w\frac{x}{3a}, J_{\delta}}v_{x}  &= \chi_{\delta}\sqbrac{w\frac{x}{3a}, J_{\delta}}\brac{\chi_{\delta'}v_{x}} 
\end{align*}
for $\delta' = \frac{\delta}{4}$. Hence, by Proposition \ref{prop:commutator}, since $\chi_{\delta'}v_{x}$ is compactly supported in $[-\frac{4}{\delta},\frac{4}{\delta}]$, we obtain that
\begin{align*}
\norm[L^{2}]{w^{b}\chi_{\delta}\sqbrac{w\frac{x}{3a}, J_{\delta}}v_{x}} &\lesssim \delta^{-\frac{1}{8}}\norm[L^{2}]{\sqbrac{w\frac{x}{3a}, J_{\delta}}\brac{\chi_{\delta'}v_{x}}}\\
 &\lesssim \delta^{\frac{1}{2}}\norm[L^{2}]{w^{\frac{1}{6}}\chi_{\delta'}v_{x}}\\
 &\lesssim \delta^{\frac{1}{2}}\brac{ \norm[L^{2}_{w}]{v} + \norm[\dot{H}^{\frac{3}{2}}]{v}}.
\end{align*}
Note that the term $\delta^{-\frac{1}{8}}$ in the first line follows from \propref{prop:weightproperties} and that $\chi_{\delta}$ is supported in $[-\frac{1}{\delta}, \frac{1}{\delta}]$. The last line follows from Proposition \ref{prop:regularityvinterpolation}.

Next, since
\begin{align*}
\norm[L^{2}]{\brac{J_{\delta}v}_{xx}} &\lesssim \delta^{-\frac{1}{2}} \norm[\dot{H}^{\frac{3}{2}}]{v},
\end{align*}
the third term in \eqref{def:Fdelta} is bounded as follows
\begin{align*}
\norm[L^{2}]{w^{b}\brac{\partial_{x}\chi_{\delta}}\brac{i\Hil J_{\delta}v}_{xx}}  &\lesssim \delta^{\frac{1}{4}} \norm[\dot{H}^{\frac{3}{2}}]{v}.
\end{align*}
The next two terms in \eqref{def:Fdelta} are bounded directly
\begin{align*}
&\norm[L^{2}]{w^{b}\brac{\partial_{x}^{2}\chi_{\delta}}\brac{i\Hil J_{\delta}v}_{x}} + \norm[L^{2}]{w^{b}\brac{\partial_{x}^{3}\chi_{\delta}}\brac{i\Hil J_{\delta}v}} \\
&\qquad \lesssim \delta \brac{\norm[L^{2}]{w^{\frac{1}{6}}v_{x}} +  \norm[L^{2}_{w}]{v}}\\
&\qquad \lesssim  \delta \brac{ \norm[L^{2}_{w}]{v} + \norm[\dot{H}^{\frac{3}{2}}]{v}}.
\end{align*}
Finally, we have the last term in \eqref{def:Fdelta} is bounded as follows.
\begin{align}
& \norm[L^{2}]{w^{b}\brac{\sqbrac{\chi_{\delta},i\Hil}J_{\delta}v}_{xxx}} \nonumber\\
&\lesssim 	\delta^{-\frac{1}{8}}\norm[L^{2}\brac{\abs{x}< \frac{3}{\delta}}]{\brac{\sqbrac{\chi_{\delta},i\Hil}J_{\delta}v}_{xxx}} + 	\norm[L^{2}\brac{\abs{x} \geq  \frac{3}{\delta}}]{w^{b}\brac{\sqbrac{\chi_{\delta},i\Hil}J_{\delta}v}_{xxx}} \nonumber\\
	&\leq 	\delta^{-\frac{1}{8}}\norm[L^{2}]{\brac{\sqbrac{\chi_{\delta},i\Hil}J_{\delta}v}_{xxx}} + 	\norm[L^{2}\brac{\abs{x} \geq  \frac{3}{\delta}}]{w^{b}\brac{\sqbrac{\chi_{\delta},i\Hil}J_{\delta}v}_{xxx}}. \label{commutator3derivativedecomp}
\end{align}
For the first term in \eqref{commutator3derivativedecomp}, we expand:
\begin{align*}
& \brac{\sqbrac{\chi_{\delta},i\Hil}J_{\delta}v}_{xxx} \\
&= \sqbrac{\partial_{x}^{3}\chi_{\delta},i\Hil}J_{\delta}v +3\sqbrac{\partial_{x}^{2}\chi_{\delta},i\Hil}J_{\delta}v_{x} +3\sqbrac{\partial_{x}\chi_{\delta},i\Hil}J_{\delta}v_{xx}+\sqbrac{\chi_{\delta},i\Hil}J_{\delta}v_{xxx}.
\end{align*}
For the latter two terms, we have
\begin{align*}
\norm[L^{2}]{\sqbrac{\partial_{x}\chi_{\delta},i\Hil}J_{\delta}v_{xx}} + \norm[L^{2}]{\sqbrac{\chi_{\delta},i\Hil}J_{\delta}v_{xxx}} &\lesssim \norm[L^{2}]{\partial_x^2\chi_{\delta}}\norm[\dot{H}^{\frac{3}{2}}]{v} \lesssim \delta^{\frac{3}{2}}\norm[\dot{H}^{\frac{3}{2}}]{v}.
\end{align*}
For the first two terms, we expand out the commutator and bound as follows:
\begin{align*}
\norm[L^{2}]{\sqbrac{\partial_{x}^{3}\chi_{\delta},i\Hil}J_{\delta}v} &\leq \norm[L^{2}]{\partial_{x}^{3}\chi_{\delta} \Hil J_{\delta}v} +\norm[L^{2}]{\Hil\brac{\partial_{x}^{3}\chi_{\delta}J_{\delta}v}}\\
&\lesssim \norm[L^{\infty}]{w^{-\frac{1}{2}}\partial_{x}^{3}\chi_{\delta}}\norm[L^{2}_{w}]{v}\\
&\lesssim \delta^{2}\norm[L^{2}_{w}]{v}
\end{align*}
and
\begin{align*}
	\norm[L^{2}]{\sqbrac{\partial_{x}^{2}\chi_{\delta},i\Hil}J_{\delta}v_{x}} &\leq \norm[L^{2}]{\partial_{x}^{2}\chi_{\delta} \Hil J_{\delta}v_{x}} +\norm[L^{2}]{\Hil\brac{\partial_{x}^{2}\chi_{\delta}J_{\delta}v_{x}}}\\
	&\lesssim \norm[L^{\infty}]{w^{-\frac{1}{6}}\partial_{x}^{2}\chi_{\delta}}\norm[L^{2}]{w^{\frac{1}{6}}v_{x}}\\
	&\lesssim \delta \norm[L^{2}]{w^{\frac{1}{6}}v_{x}}.
\end{align*}
Hence, by Proposition \ref{prop:regularityvinterpolation}, the first term in \eqref{commutator3derivativedecomp} satisfies
\begin{align*}
\delta^{-\frac{1}{8}}\norm[L^{2}]{\brac{\sqbrac{\chi_{\delta},i\Hil}J_{\delta}v}_{xxx}}  &\lesssim \delta^{\frac{1}{2}} \brac{\norm[L^{2}_{w}]{v}+\norm[\dot{H}^{\frac{3}{2}}]{v}}.
\end{align*}
For the second term in \eqref{commutator3derivativedecomp}, we write the commutator explicitly for the domain $\abs{x}\geq \frac{3}{\delta}$. If $\abs{x} \geq \frac{3}{\delta}$, then $\chi_{\delta}(x) = 0$ and so
\begin{align*}
\brac{\sqbrac{\chi_{\delta},i\Hil}J_{\delta}v}_{xxx}(x) &= - \brac{\frac{1}{\pi}\int_{\mathbb{R}} \frac{\chi_{\delta}(y)}{x-y} J_{\delta}v(y) \diff y}_{xxx}\\
&= \frac{6}{\pi}\int_{\mathbb{R}} \frac{\chi_{\delta}(y)}{(x-y)^{4}} J_{\delta}v(y) \diff y. 
\end{align*}
Note that $\chi_{\delta}(y)$ is supported in the region $\abs{y} \leq \frac{1}{\delta}$ and we are considering $\abs{x} \geq \frac{3}{\delta}$. In those regions,
 \begin{align*}
\frac{1}{\abs{x-y}} &\leq \delta
\end{align*}
and
\begin{align*}
\abs{x-y} &\geq \abs{x} - \abs{y} \geq  \frac{1}{2}\abs{x}.
\end{align*}
Thus, for $\abs{x}\geq \frac{3}{\delta}$, we have that
\begin{align*}
\abs{\brac{\sqbrac{\chi_{\delta},i\Hil}J_{\delta}v}_{xxx}(x)} \lesssim \delta^{2}\frac{1}{\abs{x}^{2}} \int_{-\frac{1}{\delta}}^{\frac{1}{\delta}} \abs{ J_{\delta}v(y)} \diff y \lesssim \frac{\delta}{\abs{x}^{2}} \norm[L^{2}_{w}]{v}. 
\end{align*}
Hence, for the second term in \eqref{commutator3derivativedecomp}, we obtain that
\begin{align*}
	\norm[L^{2}\brac{\abs{x} \geq  \frac{3}{\delta}}]{w^{b}\brac{\sqbrac{\chi_{\delta},i\Hil}J_{\delta}v}_{xxx}} \lesssim \delta  \norm[L^{2}_{w}]{v} 	\norm[L^{2}\brac{\abs{x} \geq  \frac{3}{\delta}}]{w^{b}\abs{x}^{-2}} \lesssim \delta  \norm[L^{2}_{w}]{v}.
\end{align*}
Thus, we have for $b=0$ or $b=\frac{1}{2}$,
\begin{align*}
	\norm[L^{2}]{w^{b}\brac{\sqbrac{\chi_{\delta},i\Hil}J_{\delta}v}_{xxx}} &\lesssim \delta^{\frac{1}{2}} \brac{\norm[L^{2}_{w}]{v}+\norm[\dot{H}^{\frac{3}{2}}]{v}}.
\end{align*}
Collecting all of the estimates concludes the proof of this proposition.
\end{proof}

We can now prove the desired decay estimate.

\begin{prop}\label{prop:gainofdecay}
If $v\in L^{2}_{w}\brac{\mathbb{R}}\cap \dot{H}^{\frac{3}{2}}\brac{\mathbb{R}}$ satisfies \eqref{eq:linearequation} in a distributional sense for $F, w^{-\frac{1}{2}}F \in L^{2}\brac{\mathbb{R}}$, then $xv_{x}\in L^{2}_{w}\brac{\mathbb{R}}$ and moreover,
\begin{align*}
	\norm[L^{2}_{w}]{xv_{x}} &\lesssim \norm[L^{2}]{F} + \norm[L^{2}]{w^{-\frac{1}{2}}F} + \norm[L^{2}_{w}]{v} + \norm[\dot{H}^{\frac{3}{2}}]{v}.
\end{align*}
\end{prop}

\begin{proof}
Since $v\in L^{2}_{w}(\mathbb{R})\cap\dot{H}^{\frac{3}{2}}(\mathbb{R})$ satisfies \cref{eq:linearequation} in a distributional sense, we have that \eqref{eq:lineardelta} holds pointwise. From \eqref{eq:lineardelta}, we obtain that
\begin{align*}
\int_{\mathbb{R}} x\brac{\chi_{\delta}J_{\delta}v}_{x}\mathcal{L}\brac{\chi_{\delta}J_{\delta}v} \diff x  &= \int_{\mathbb{R}} \brac{ x\brac{\chi_{\delta}J_{\delta}v}_{x}F_{\delta} + w\frac{x^{2}}{3a}\brac{\chi_{\delta}J_{\delta}v}_{x}\brac{\partial_{x}\chi_{\delta}}J_{\delta}v} \diff x .
\end{align*}
For the first term on the right hand side, we have that
\begin{align*}
 \int_{\mathbb{R}} x\brac{\chi_{\delta}J_{\delta}v}_{x}F_{\delta} \diff x &\lesssim\norm[L^{2}_{w}]{x\brac{\chi_{\delta}J_{\delta}v}_{x}}\norm[L^{2}]{w^{-\frac{1}{2}}F_{\delta}}.
\end{align*}
Next, using integration by parts,
\begin{align*}
& \int_{\mathbb{R}} w\frac{x^{2}}{3a}\brac{\chi_{\delta}J_{\delta}v}_{x}\brac{\partial_{x}\chi_{\delta}}J_{\delta}v \diff x \\
 &= - \int_{\mathbb{R}} \brac{w\frac{x^{2}}{3a}}_{x}\chi_{\delta}\brac{\partial_{x}\chi_{\delta}}(J_{\delta}v)^{2} \diff x - \int_{\mathbb{R}} w\frac{x^{2}}{3a}\chi_{\delta}\brac{\partial_{x}^{2}\chi_{\delta}}(J_{\delta}v)^{2} \diff x\\
&\quad  - \int_{\mathbb{R}} w\frac{x^{2}}{3a}\chi_{\delta}J_{\delta}v\brac{\partial_{x}\chi_{\delta}}\brac{J_{\delta}v}_{x} \diff x .
\end{align*}
Using that $\abs{x}\leq \frac{1}{\delta}$ in the support of $\chi_{\delta}$ with \cref{prop:weightproperties} and  \cref{prop:JdeltaHilnorm},
\begin{align*}
\abs{\int_{\mathbb{R}} \brac{w\frac{x^{2}}{3a}}_{x}\chi_{\delta}\brac{\partial_{x}\chi_{\delta}}(J_{\delta}v)^{2} \diff x } +\abs{\int_{\mathbb{R}} w\frac{x^{2}}{3a}\chi_{\delta}\brac{\partial_{x}^{2}\chi_{\delta}}(J_{\delta}v)^{2} \diff x }
 &\lesssim  \norm[L^{2}_{w}]{v}^{2}.
\end{align*}
Next, integrating by parts, we obtain
\begin{align*}
& \int_{\mathbb{R}} w\frac{x^{2}}{3a}\chi_{\delta}J_{\delta}v\brac{\partial_{x}\chi_{\delta}}\brac{J_{\delta}v}_{x} \diff x \\
&= \frac{1}{2} \int_{\mathbb{R}} w\frac{x^{2}}{3a}\chi_{\delta}\brac{\partial_{x}\chi_{\delta}}\brac{\brac{J_{\delta}v}^{2}}_{x} \diff x \\
&= - \frac{1}{2} \int_{\mathbb{R}}\brac{ w\frac{x^{2}}{3a}\brac{\chi_{\delta}\brac{\partial_{x}\chi_{\delta}}}_{x} + \brac{w\frac{x^{2}}{3a}}_{x}\chi_{\delta}\brac{\partial_{x}\chi_{\delta}}}\brac{J_{\delta}v}^{2} \diff x
\end{align*}
and hence, similarly to the previous estimate,
\begin{align*}
	\abs{\int_{\mathbb{R}} w\frac{x^{2}}{3a}\chi_{\delta}J_{\delta}v\brac{\partial_{x}\chi_{\delta}}\brac{J_{\delta}v}_{x} \diff x} 
	&\lesssim \norm[L^{2}_{w}]{v}^{2}.
\end{align*}
In summary, we have shown that
\begin{align*}
	\int_{\mathbb{R}} x\brac{\chi_{\delta}J_{\delta}v}_{x}\mathcal{L}\brac{\chi_{\delta}J_{\delta}v} \diff x  &\lesssim \norm[L^{2}_{w}]{x\brac{\chi_{\delta}J_{\delta}v}_{x}}\norm[L^{2}]{w^{-\frac{1}{2}}F_{\delta}} +  \norm[L^{2}_{w}]{v}^{2}.
\end{align*}
For convenience of notation, let us call $\chi_{\delta}J_{\delta}v = \zeta$. Then, the left hand side of the above inequality is
\begin{align*}
\int_\Rsp x \zeta_{x} \mathcal{L}(\zeta) \diff x &= \int_\Rsp x \zeta_{x} i\Hil\zeta_{xxx} \diff x + \frac{1}{3a} \norm[L^{2}_{w}]{x\zeta_{x}}^{2}.
\end{align*}
For the first term, we integrate by parts to obtain
\begin{align*}
\int_\Rsp x \zeta_{x} i\Hil\zeta_{xxx} \diff x  &= -\int_\Rsp  \zeta_{x} i\Hil\zeta_{xx} \diff x -\int_\Rsp x \zeta_{xx} i\Hil\zeta_{xx} \diff x\\
&= -\norm[\dot{H}^{\frac{3}{2}}]{\zeta}^{2} -\int_\Rsp x \zeta_{xx} i\Hil\zeta_{xx} \diff x.
\end{align*}
The latter integral is zero since
\begin{align*}
 \int_\Rsp x \zeta_{xx} i\Hil\zeta_{xx} \diff x & = -\int_\Rsp i\Hil\brac{ x \zeta_{xx}}\zeta_{xx} \diff x \\
 &=  -\int_\Rsp\brac{\sqbrac{i\Hil, x} \zeta_{xx}}\zeta_{xx} \diff x -  \int_\Rsp  x \brac{i\Hil\zeta_{xx}}\zeta_{xx} \diff x\\
 &=  -  \int_\Rsp  x \brac{i\Hil\zeta_{xx}}\zeta_{xx} \diff x.
\end{align*}
Above, we have used that $\sqbrac{i\Hil,x}f_{x} = 0$ for any compactly supported smooth function $f$.
In summary, we have proven that
\begin{align*}
-\norm[\dot{H}^{\frac{3}{2}}]{\chi_{\delta}J_{\delta}v}^{2} +  \frac{1}{3a} \norm[L^{2}_{w}]{x\brac{\chi_{\delta}J_{\delta}v}_{x}}^{2} &\lesssim \norm[L^{2}_{w}]{v}^{2} +  \norm[L^{2}_{w}]{x\brac{\chi_{\delta}J_{\delta}v}_{x}}\norm[L^{2}]{w^{-\frac{1}{2}}F_{\delta}}
\end{align*}
and thus, for all $0< \delta \leq 1$, we have that
\begin{align*}
	  \norm[L^{2}_{w}]{x\brac{\chi_{\delta}J_{\delta}v}_{x}}^{2} &\lesssim \norm[L^{2}_{w}]{v}^{2} + \norm[\dot{H}^{\frac{3}{2}}]{\chi_{\delta}J_{\delta}v}^{2} +\norm[L^{2}]{w^{-\frac{1}{2}}F_{\delta}}^{2}.
\end{align*}
Proposition \ref{prop:chideltaJdeltanorm} and Proposition \ref{prop:Fdeltaconvergence} imply that
\begin{align}\label{ineq:deltadecay}
 \norm[L^{2}_{w}]{x\brac{\chi_{\delta}J_{\delta}v}_{x}}^{2} &\lesssim \norm[\dot{H}^{\frac{3}{2}}]{v}^{2} + \norm[L^{2}_{w}]{v}^{2} + \norm[L^{2}]{F}^{2} + \norm[L^{2}]{w^{-\frac{1}{2}}F}^{2}.
\end{align}
We now compute the left hand side of the above inequality. We have for $0< \delta \leq 1$,
\begin{align*}
\norm[L^{2}_{w}]{x\chi_{\delta}J_{\delta}v_{x}} &\leq \norm[L^{2}_{w}]{x\brac{\partial_{x}\chi_{\delta}}J_{\delta}v} + \norm[L^{2}_{w}]{x\brac{\chi_{\delta}J_{\delta}v}_{x}}\\
&\lesssim \norm[L^{2}_{w}]{v} + \norm[L^{2}_{w}]{x\brac{\chi_{\delta}J_{\delta}v}_{x}}
\end{align*}
and thus, by \eqref{ineq:deltadecay}, $\norm[L^{2}_{w}]{x\chi_{\delta}J_{\delta}v_{x}}$ is uniformly bounded for $0 < \delta \leq 1$:
\begin{align*}
\norm[L^{2}_{w}]{x\chi_{\delta}J_{\delta}v_{x}} &\lesssim  \norm[\dot{H}^{\frac{3}{2}}]{v} + \norm[L^{2}_{w}]{v} + \norm[L^{2}]{F} + \norm[L^{2}]{w^{-\frac{1}{2}}F}.
\end{align*} 
By Proposition \ref{prop:deltaconvergenceL2}, we obtain that $xv_{x}\in L^{2}_{w}(\mathbb{R})$ and
\begin{align*}
	\norm[L^{2}_{w}]{xv_{x}} &\lesssim \norm[\dot{H}^{\frac{3}{2}}]{v} + \norm[L^{2}_{w}]{v} +\norm[L^{2}]{F} + \norm[L^{2}]{w^{-\frac{1}{2}}F}.
\end{align*}
\end{proof}

We gain the following control:
\begin{prop}\label{cor:decayinterpolation}
If If $v\in L^{2}_{w}(\mathbb{R}) \cap \dot{H}^{\frac{3}{2}}(\mathbb{R})$ and $xv_{x}\in L^{2}_{w}(\mathbb{R})$, then $v\in\dot{H}^{1}(\mathbb{R})$. Moreover, if $\abs{b}\leq 10$, then 
\begin{align*}
\norm[L^{2}]{w^{b}v_{x}} + \norm[L^2_{w}]{ (1+x^2)^\half v_x} &\lesssim  \norm[L^{2}_{w}]{v}  +\norm[L^{2}_{w}]{x v_{x}}.
\end{align*}
\end{prop}
\begin{proof}
By Proposition \ref{prop:regularityvinterpolation}, we have that $w^{\frac{1}{6}}v_{x}\in L^{2}(\mathbb{R})$. Hence,
\begin{align*}
\norm[L^{2}]{w^{b}v_{x}} &\leq \norm[L^{2}(\abs{x}<1)]{w^{b}v_{x}} +\norm[L^{2}(\abs{x}\geq 1)]{w^{b}v_{x}} \\
&\lesssim \norm[L^{\infty}(\abs{x}<1)]{w^{-\frac{1}{6}+b}}\norm[L^{2}(\abs{x}<1)]{w^{\frac{1}{6}}v_{x}} +\norm[L^{\infty}(\abs{x}\geq 1)]{w^{-\frac{1}{2}+b}x^{-1}}\norm[L^{2}(\abs{x}\geq 1)]{w^{\frac{1}{2}}x v_{x}}\\
&\lesssim \norm[L^{2}]{w^{\frac{1}{6}}v_{x}} +\norm[L^{2}_{w}]{x v_{x}}
\end{align*}
Using Proposition \ref{prop:regularityvinterpolation}, we conclude the proof.
\end{proof}

We can now prove the gain of regularity estimate.

\begin{prop}\label{prop:gainofregularity}
	Suppose $v\in L^{2}_{w}(\mathbb{R})\cap\dot{H}^{\frac{3}{2}}(\mathbb{R})$ satisfies \eqref{eq:linearequation} in a distributional sense. If $F, w^{-\frac{1}{2}}F\in L^{2}(\mathbb{R})$, then $v\in\dot{H}^{3}(\mathbb{R})$. Moreover,
	\begin{align*}
		\norm[\dot{H}^{3}]{v} &\lesssim \norm[L^{2}_{w}]{v} + \norm[\dot{H}^{\frac{3}{2}}]{v} + \norm[L^{2}_{w}]{xv_{x}} + \norm[L^{2}]{F} + \norm[L^{2}]{w^{-\frac{1}{2}}F}.
	\end{align*}
\end{prop}
\begin{proof}
	Since $v\in L^{2}_{w}(\mathbb{R})\cap\dot{H}^{\frac{3}{2}}(\mathbb{R})$ satisfies \eqref{eq:linearequation} in a distributional sense, we have that \eqref{eq:lineardelta} holds pointwise. Moreover Proposition \ref{prop:gainofdecay} and Proposition \ref{cor:decayinterpolation} implies that $v\in \dot{H}^{1}(\mathbb{R})$ and $xv_{x}\in L^{2}_{w}(\mathbb{R})$.
	
	Multiplying both sides of \eqref{eq:lineardelta} by $\abs{D}^{3}\brac{\chi_{\delta}J_{\delta}v}$, we obtain that
	\begin{align*}
		\int_{\mathbb{R}} 	\abs{D}^{3}\brac{\chi_{\delta}J_{\delta}v}\mathcal{L}\brac{\chi_{\delta}J_{\delta}v} \diff x= \int_{\mathbb{R}} 	\abs{D}^{3}\brac{\chi_{\delta}J_{\delta}v} F_{\delta} \diff x + \int_{\mathbb{R}} 	\abs{D}^{3}\brac{\chi_{\delta}J_{\delta}v} w\frac{x}{3a}\brac{\partial_{x}\chi_{\delta}}J_{\delta}v\diff x.
	\end{align*}
	For the first term on the right hand side, we have that
	\begin{align*}
	\abs{\int_{\mathbb{R}} 	\abs{D}^{3}\brac{\chi_{\delta}J_{\delta}v} F_{\delta} \diff x} &\leq \norm[\dot{H}^{3}]{\chi_{\delta}J_{\delta}v}\norm[L^{2}]{F_{\delta}}.
	\end{align*}
	For the second term on the right hand side, we have
	\begin{align*}
		\int_{\mathbb{R}} 	\abs{D}^{3}\brac{\chi_{\delta}J_{\delta}v} w\frac{x}{3a}\brac{\partial_{x}\chi_{\delta}}J_{\delta}v\diff x &=  -\int_{\mathbb{R}} 	i\Hil\abs{D}^{2}\brac{\chi_{\delta}J_{\delta}v}\partial_{x}\brac{ w\frac{x}{3a}\brac{\partial_{x}\chi_{\delta}}J_{\delta}v}\diff x \\
		&= -J_{1} - J_{2}
	\end{align*}
	where using the support of $\chi_{\delta}$ and $\phi$ in the definition of \eqref{def:Jdelta},
		\begin{align*}
		J_{1} &= \int_{\mathbb{R}} 	i\Hil\abs{D}^{2}\brac{\chi_{\delta}J_{\delta}v} w\frac{x}{3a}\brac{\partial_{x}\chi_{\delta}}J_{\delta}v_{x}\diff x
	\end{align*}
 and
		\begin{align*}
		J_{2} &= \int_{\mathbb{R}} 	i\Hil\abs{D}^{2}\brac{\chi_{\delta}J_{\delta}v}\partial_{x}\brac{ w\frac{x}{3a}\brac{\partial_{x}\chi_{\delta}}}J_{\delta}v\diff x.
	\end{align*}
Using Proposition \ref{cor:decayinterpolation} and Proposition \ref{prop:JdeltaHilnorm}, we have
	\begin{align*}
		\abs{J_{1}}
		&\lesssim \norm[\dot{H}^{2}]{ \chi_{\delta}J_{\delta}v}\norm[L^{\infty}]{x\partial_{x}\chi_{\delta}} \norm[L^{2}]{w J_{\delta}v_{x}}\\
		&\lesssim \norm[\dot{H}^{3}]{\chi_{\delta}J_{\delta}v}^{\frac{1}{3}}\norm[\dot{H}^{\frac{3}{2}}]{\chi_{\delta}J_{\delta}v}^{\frac{2}{3}}\brac{\norm[L^{2}_{w}]{x v_{x}} + \norm[L^{2}_{w}]{v} }.
	\end{align*}
	and similarly,
	\begin{align*}
		\abs{J_{2}} &\lesssim \norm[\dot{H}^{2}]{ \chi_{\delta}J_{\delta}v}\norm[L^{\infty}]{w^{-\frac{1}{2}}\partial_{x}\brac{ w\frac{x}{3a}\brac{\partial_{x}\chi_{\delta}}}} \norm[L^{2}_{w}]{J_{\delta}v} \\
		&\lesssim \norm[\dot{H}^{3}]{\chi_{\delta}J_{\delta}v}^{\frac{1}{3}}\norm[\dot{H}^{\frac{3}{2}}]{\chi_{\delta}J_{\delta}v}^{\frac{2}{3}}\norm[L^{2}_{w}]{v}
	\end{align*}
	where we used that $x w_{x} = 3i\ep  x \Hil G_{x} w $ and $x\Hil G_{x} \in L^{\infty}(\mathbb{R})$ by Proposition \ref{prop:Gproperties}.
	In summary,
	\begin{align}\label{est:J1J2}
		\abs{\int_{\mathbb{R}} \abs{D}^{3}\brac{\chi_{\delta}J_{\delta}v} w\frac{x}{3a}\brac{\partial_{x}\chi_{\delta}}J_{\delta}v\diff x}
		& \lesssim \norm[\dot{H}^{3}]{\chi_{\delta}J_{\delta}v}^{\frac{1}{3}}\norm[\dot{H}^{\frac{3}{2}}]{\chi_{\delta}J_{\delta}v}^{\frac{2}{3}}\brac{\norm[L^{2}_{w}]{x v_{x}} + \norm[L^{2}_{w}]{v} }.
	\end{align}

Next, we deal with the left hand side. Define $\zeta = \chi_{\delta}J_{\delta}v$ for convenience of notation. Then,
	\begin{align*}
		\int_{\mathbb{R}} \brac{\abs{D}^{3}\zeta}\brac{ \mathcal{L}\brac{\zeta}} \diff x &= -\norm[\dot{H}^{3}]{\zeta}^{2} + \int_{\mathbb{R}}\brac{\abs{D}^{3}\zeta} w \frac{x}{3a}\zeta_{x} \diff x.
	\end{align*}
	Integrating by parts, we obtain that
	\begin{align*}
		\abs{\int_{\mathbb{R}}\brac{\abs{D}^{3}\zeta} w \frac{x}{3a}\zeta_{x} \diff x} &=\abs{ \int_{\mathbb{R}}\brac{i\Hil\abs{D}^{2}\zeta} \partial_{x}\brac{w \frac{x}{3a}\zeta_{x}} \diff x}\\
		& \leq \abs{ \int_{\mathbb{R}}\brac{i\Hil\abs{D}^{2}\zeta} \partial_{x}\brac{w \frac{x}{3a}}\zeta_{x} \diff x} +\abs{ \int_{\mathbb{R}}\brac{i\Hil\zeta_{xx}} w \frac{x}{3a}\zeta_{xx}  \diff x}
	\end{align*}
	The first term is controlled in the same way as $J_{1}$ earlier:
	\begin{align*}
	\abs{ \int_{\mathbb{R}}\brac{i\Hil\abs{D}^{2}\zeta} \partial_{x}\brac{w \frac{x}{3a}}\zeta_{x} \diff x}	&\lesssim \norm[\dot{H}^{2}]{\zeta}\norm[L^{2}]{\partial_{x}{\brac{w \frac{x}{3a}}} \zeta_{x}}\\
	&\lesssim \norm[\dot{H}^{3}]{\chi_{\delta}J_{\delta}v}^{\frac{1}{3}}\norm[\dot{H}^{\frac{3}{2}}]{\chi_{\delta}J_{\delta}v}^{\frac{2}{3}}\brac{\norm[L^{2}_{w}]{x v_{x}} + \norm[L^{2}_{w}]{v} }
	\end{align*}	
	The second term, call it $K$, is controlled as follows:
	\begin{align*}
		K &= \int_{\mathbb{R}}\brac{i\Hil\zeta_{xx}} w \frac{x}{3a}\zeta_{xx}  \diff x\\
		&= -\int_{\mathbb{R}}\zeta_{xx} i\Hil\brac{ w \frac{x}{3a}\zeta_{xx}}  \diff x\\
		&= -K - \int_{\mathbb{R}}\zeta_{xx}\sqbrac{i\Hil,w \frac{x}{3a}}\zeta_{xx}  \diff x
	\end{align*}
	Hence,
	\begin{align*}
		\abs{K} &\lesssim \abs{\int_{\mathbb{R}}\zeta_{xx}\sqbrac{i\Hil,w \frac{x}{3a}}\zeta_{xx}  \diff x}\\
		&= \abs{\int_{\mathbb{R}}\zeta_{x}\partial_{x}\brac{\sqbrac{i\Hil,w \frac{x}{3a}}\zeta_{xx}}  \diff x}\\
		&\lesssim \norm[\dot{H}^{1}]{\zeta}^{2}\norm[L^{\infty}]{\partial_{x}^{2}\brac{wx}}\\
		&\lesssim \norm[L^{2}_{w}]{v}^{2} + \norm[L^{2}_{w}]{x v_{x}}^{2}.
	\end{align*}
Then, collecting all of the above estimates, we have
\begin{align*}
	\norm[\dot{H}^{3}]{\chi_{\delta}J_{\delta}v}^{2} &\lesssim \norm[L^{2}_{w}]{v}^{2} + \norm[L^{2}_{w}]{x v_{x}}^{2} + \norm[\dot{H}^{3}]{\chi_{\delta}J_{\delta}v}^{\frac{1}{3}}\norm[\dot{H}^{\frac{3}{2}}]{\chi_{\delta}J_{\delta}v}^{\frac{2}{3}}\brac{\norm[L^{2}_{w}]{x v_{x}} + \norm[L^{2}_{w}]{v} } \\
	& \quad + \norm[\dot{H}^{3}]{\chi_{\delta}J_{\delta}v}\norm[L^{2}]{F_{\delta}}
\end{align*}
By Young's inequality, we obtain that
	\begin{align*}
	\norm[\dot{H}^{3}]{\chi_{\delta}J_{\delta}v}^{2} &\lesssim  \norm[\dot{H}^{\frac{3}{2}}]{\chi_{\delta}J_{\delta}v}^{2} + \norm[L^{2}_{w}]{v}^{2} + \norm[L^{2}_{w}]{xv_{x}}^{2} + \norm[L^{2}]{F_{\delta}}^{2}.
	\end{align*}
Since
\begin{align*}
	\norm[L^{2}]{\chi_{\delta}J_{\delta}v_{xxx}} &\lesssim 	\norm[\dot{H}^{3}]{\chi_{\delta}J_{\delta}v} + \norm[L^{2}]{\brac{\partial_{x}\chi_{\delta}}J_{\delta}v_{xx}} + \norm[L^{2}]{\brac{\partial_{x}^{2}\chi_{\delta}}J_{\delta}v_{x}} + \norm[L^{2}]{\brac{\partial_{x}^{3}\chi_{\delta}}J_{\delta}v}\\
	&\lesssim \norm[\dot{H}^{3}]{\chi_{\delta}J_{\delta}v}  + \delta^{\frac{1}{2}}\brac{\norm[\dot{H}^{\frac{3}{2}}]{v} +  \norm[L^{2}_{w}]{v}} 
\end{align*}
Thus, additionally using Proposition \ref{prop:chideltaJdeltanorm}, Proposition \ref{prop:Fdeltaconvergence}, we obtain that
\begin{align*}
	\norm[L^{2}]{\chi_{\delta}J_{\delta}v_{xxx}} &\lesssim \norm[\dot{H}^{\frac{3}{2}}]{v} + \norm[L^{2}_{w}]{v} + \norm[L^{2}_{w}]{xv_{x}} + \norm[L^{2}]{F} + \norm[L^{2}]{w^{-\frac{1}{2}}F}.
\end{align*}
Now, let $0<\tau\leq 1$. Using \cref{prop:appendixJdelta} we get
\begin{align*}
\norm[L^{2}]{\chi_{\delta}J_{\delta}J_{\tau}v_{xxx}}  &\leq  \norm[L^{2}]{J_{\tau}\brac{\chi_{\delta}J_{\delta}v_{xxx}}} + \norm[L^{2}]{\sqbrac{J_{\tau},\chi_{\delta}}J_{\delta}v_{xxx}} \lesssim \norm[L^{2}]{\chi_{\delta}J_{\delta}v_{xxx}} + \norm[\dot{H}^{\frac{3}{2}}]{v}.
\end{align*}
By Proposition \ref{prop:deltaconvergenceL2} and the above estimates, we have the following uniform bound for all $0<\tau\leq 1$:
\begin{align*}
	\norm[L^{2}]{J_{\tau}v_{xxx}}  &\lesssim \norm[\dot{H}^{\frac{3}{2}}]{v} + \norm[L^{2}_{w}]{v} + \norm[L^{2}_{w}]{xv_{x}} + \norm[L^{2}]{F} + \norm[L^{2}]{w^{-\frac{1}{2}}F}.
\end{align*}
Letting $\tau$ go to zero concludes the proof.
\end{proof}

\subsection{Uniqueness of the Linear Solution}\label{subsec:uniq}
In this section, we will show that the solution $v\in L^{2}_{w}(\mathbb{R})\cap\dot{H}^{\frac{3}{2}}(\mathbb{R})$ to \eqref{eq:linearequation} is unique. To do so, we first need the following lemma regarding the second term on the right hand side of \eqref{eq:lineardelta}:

\begin{lem}\label{lem:eq41term}
If $v\in L^{2}_{w}(\mathbb{R})$, then 
\begin{align*}
\lim_{\delta\rightarrow 0} \norm[L^{2}_{w}]{x\brac{\partial_{x}\chi_{\delta}}J_{\delta}v} = 0.
\end{align*}
\end{lem}
\begin{proof}
Since $\abs{x\partial_{x}\chi_{\delta}}\leq 1$, it follows that  $w^{\frac{1}{2}}x\brac{\partial_{x}\chi_{\delta}}J_{\delta}v$ is uniformly dominated by $w^{\frac{1}{2}}M\brac{v}$, where $M$ is the Hardy-Littlewood maximal operator. Because  $w^{\frac{1}{2}}M\brac{v} \in L^{2}(\mathbb{R})$ by Proposition \ref{prop:JdeltaHilnorm}, we conclude the proof by the Dominated Convergence Theorem.
\end{proof}

The following proposition implies that the solution $v\in L^{2}_{w}(\mathbb{R}) \cap \dot{H}^{1}(\mathbb{R})\cap \dot{H}^{3}(\mathbb{R})$ is unique:
\begin{prop}\label{prop:uniqueness}
	If $v\in L^{2}_{w}(\mathbb{R})\cap \dot{H}^{\frac{3}{2}}(\mathbb{R})$ is a solution to \eqref{eq:linearequation} and $F, w^{-\frac{1}{2}}F \in L^{2}\brac{\mathbb{R}}$, then 
	\begin{align*}
	\norm[\dot{H}^{\frac{3}{2}}]{v} + \norm[L^{2}_{w}]{v} &\lesssim \norm[L^2]{F} + \norm[L^{2}]{w^{-\frac{1}{2}}F}
	\end{align*}
	and therefore, the solution $v\in L^{2}_{w}(\mathbb{R})\cap \dot{H}^{\frac{3}{2}}(\mathbb{R})$ to \eqref{eq:linearequation} is unique.
\end{prop}
\begin{proof}
By Proposition \ref{prop:gainofdecay}, \cref{cor:decayinterpolation} and Proposition \ref{prop:gainofregularity}, we have that $v\in L^{2}_{w}(\mathbb{R})\cap\dot{H}^{1}(\mathbb{R})\cap\dot{H}^{3}(\mathbb{R})$ and $xv_{x}\in L^{2}_{w}(\mathbb{R})$.

Multiply both sides of \eqref{eq:lineardelta} by $\chi_{\delta}J_{\delta}v$ and integrate to obtain that
\begin{align*}
\int_{\mathbb{R}}\chi_{\delta}J_{\delta}v\mathcal{L}\brac{\chi_{\delta}J_{\delta}v} \diff x &= \int_{\mathbb{R}}\chi_{\delta}J_{\delta}vF_{\delta} \diff x +\int_{\mathbb{R}}\chi_{\delta}J_{\delta}v w\frac{x}{3a}\brac{\partial_{x}\chi_{\delta}}J_{\delta}v \diff x.
\end{align*}
We have that
\begin{align*}
\int_{\mathbb{R}}\chi_{\delta}J_{\delta}v\mathcal{L}\brac{\chi_{\delta}J_{\delta}v} \diff x &= -\norm[\dot{H}^{\frac{3}{2}}]{\chi_{\delta}J_{\delta}v}^{2} - \frac{1}{6a}\int_{\mathbb{R}} \brac{xw}_{x} \brac{\chi_{\delta}J_{\delta}v}^{2} \diff x
\end{align*}
By Proposition \ref{prop:weightproperties}, we obtain that
\begin{align*}
\norm[\dot{H}^{\frac{3}{2}}]{\chi_{\delta}J_{\delta}v}^{2} + \norm[L^{2}_{w}]{\chi_{\delta}J_{\delta}v}^{2} &\lesssim \norm[L^{2}_{w}]{\chi_{\delta}J_{\delta}v}\brac{\norm[L^{2}]{w^{-\frac{1}{2}}F_{\delta}} +  \norm[L^{2}_{w}]{x \brac{\partial_{x}\chi_{\delta}}J_{\delta}v}}.
\end{align*}
By Young's inequality and Proposition \ref{prop:Fdeltaconvergence}, we get
\begin{align*}
\norm[\dot{H}^{\frac{3}{2}}]{\chi_{\delta}J_{\delta}v} + \norm[L^{2}_{w}]{\chi_{\delta}J_{\delta}v} &\lesssim \delta^{\frac{1}{4}}\brac{\norm[L^{2}_{w}]{v} + \norm[\dot{H}^{\frac{3}{2}}]{v}} + \norm[L^{2}]{F} + \norm[L^{2}]{w^{-\frac{1}{2}}F} \\
& \quad +  \norm[L^{2}_{w}]{x \brac{\partial_{x}\chi_{\delta}}J_{\delta}v}.
\end{align*}
Letting $\delta \to 0$ and using Proposition \ref{prop:chideltaJdeltanorm} and Lemma \ref{lem:eq41term}, we conclude the proof.
\end{proof}

\subsection{Further Properties of the Linear Solution}\label{subsec:furtherproperties}

We will now prove some further properties of the linear solution $v\in L^{2}_{w}(\mathbb{R})\cap\dot{H}^{1}(\mathbb{R})\cap \dot{H}^{3}(\mathbb{R})$ using that $xv_x \in L^{2}_{w}(\mathbb{R})$.

\begin{prop}\label{prop:furtherproperties}
Let $\ep_{1} > 0$ be given by \cref{prop:weightproperties} and let $0 \leq \abs{\ep} \leq \ep_1$. Let $w$ be defined by \eqref{def:w} and let $v\in L^{2}_{w}(\mathbb{R})\cap \dot{H}^{1}(\mathbb{R}) \cap \dot{H}^{3}(\mathbb{R})$ and $xv_{x}\in L^{2}_{w}(\mathbb{R})$. Then, we have the following additional regularity and decay properties:
\begin{enumerate}[leftmargin =*, align = left]
\item We have $v, v_{x}, v_{xx} \in L^{\infty}(\mathbb{R})$ and 
\begin{align*}
& \norm[L^2_w]{v} + \norm[L^{\infty}]{(1+x^{2})^{\frac{1}{8}}v} +  \norm[L^2_{w}]{ (1+x^2)^\half v_x}  + \norm[L^{\infty}]{\brac{1+x^{2}}^{\frac{1}{8}}v_{x}}   \\
&  + \norm[L^{2}]{(1+x^{2})^{\frac{1}{16}}v_{xx}} + \norm[L^{\infty}]{(1+x^{2})^{\frac{1}{32}}v_{xx}} + \norm[L^{2}]{v_{xxx}} \\
&  \lesssim \norm[X]{v}
\end{align*}
where $\norm[X]{v}$ is defined in \cref{def:Xnorm}. 
\item We have that $\norm[X]{\Hil v} \lesssim \norm[X]{v}$ and hence the same estimate as above holds with $v$ replaced by $\Hil v$ in the left hand side of the above estimate.
\end{enumerate}
\end{prop}
\begin{proof}
We prove the statements individually.
\begin{enumerate}
\item We have $w^{\frac{1}{2}}v\in L^{2}(\mathbb{R})$ and by  \cref{cor:decayinterpolation}, we have $(1+x^{2})^{\frac{1}{2}}w^{\frac{1}{2}}v_{x}\in L^{2}(\mathbb{R})$.  \cref{lem:infinitydecay} now implies that 
\begin{align*}
\norm[L^{\infty}]{\brac{1+x^{2}}^{\frac{1}{4}}w^{\frac{1}{2}}v} &\lesssim \norm[L^{2}_{w}]{v}+\norm[L^{2}_{w}]{(1+x^{2})^{\frac{1}{2}}v_{x}}.
\end{align*}
and $\brac{1+x^{2}}^{\frac{1}{4}}w^{\frac{1}{2}}v \rightarrow 0$ as $\abs{x}\rightarrow \infty$. Thus by \cref{prop:weightproperties} we obtain that
\begin{align*}
\norm[L^{\infty}]{\brac{1+x^{2}}^{\frac{1}{8}}v} &\lesssim \norm[X]{v}.
\end{align*}
Similarly, since $v_{xx}\in L^{2}(\mathbb{R})$, Lemma \ref{lem:infinitydecay} implies that
\begin{align*}
\norm[L^{\infty}]{\brac{1+x^{2}}^{\frac{1}{4}}w^{\frac{1}{4}}v_{x}} &\lesssim \norm[L^{2}_{w}]{(1+x^{2})^{\frac{1}{2}}v_{x}} + \norm[L^{2}]{v_{xx}} \lesssim \norm[X]{v}
\end{align*}
and $\brac{1+x^{2}}^{\frac{1}{4}}w^{\frac{1}{4}}v_{x} \rightarrow 0$ as $\abs{x}\rightarrow \infty$. Hence
\begin{align*}
\norm[L^{\infty}]{\brac{1+x^{2}}^{\frac{1}{8}}v_{x}} &\lesssim  \norm[X]{v}
\end{align*}
Next, 
\begin{align*}
\norm[L^{\infty}]{v_{xx}} &\lesssim \norm[H^{1}]{v_{xx}} \lesssim \norm[X]{v}.
\end{align*}
Using integration by parts and the fact that $\brac{1+x^{2}}^{\frac{1}{4}}w^{\frac{1}{4}}v_{x} \rightarrow 0$ as $\abs{x}\rightarrow \infty$, we obtain
\begin{align*}
\norm[L^{2}]{(1+x^{2})^{\frac{1}{8}}w^{\frac{1}{8}}v_{xx}}^{2} 
&= \int (1+x^{2})^{\frac{1}{4}}w^{\frac{1}{4}} v_{xx}^{2} \diff x \\
&= -\int \partial_{x}\brac{(1+x^{2})^{\frac{1}{4}}w^{\frac{1}{4}}} v_{x}v_{xx} \diff x - \int (1+x^{2})^{\frac{1}{4}}w^{\frac{1}{4}} v_{x} v_{xxx} \diff x.
\end{align*}
The first term satisfies
\begin{align*}
\abs{\int \partial_{x}\brac{(1+x^{2})^{\frac{1}{4}}w^{\frac{1}{4}}} v_{x}v_{xx} \diff x}
&\lesssim \norm[\dot{H}^{1}]{v}\norm[\dot{H}^{2}]{v} \lesssim \norm[X]{v}^{2}.
\end{align*}
For the second term, by \cref{cor:decayinterpolation}, we have that
\begin{align*}
&\abs{ \int (1+x^{2})^{\frac{1}{4}}w^{\frac{1}{4}} v_{x} v_{xxx} \diff x}\\
&\lesssim  \norm[L^{\infty}]{w^{-\frac{1}{4}}(1+x^{2})^{-\frac{1}{4}}}\norm[L^{2}_{w}]{(1+x^{2})^{\frac{1}{2}}v_{x}}\norm[\dot{H}^{3}]{v}\\
&\lesssim \norm[X]{v}^{2}.
\end{align*}
Thus, we obtain that
\begin{align*}
\norm[L^{2}]{(1+x^{2})^{\frac{1}{16}}v_{xx}} &\lesssim 	\norm[L^{2}]{(1+x^{2})^{\frac{1}{8}}w^{\frac{1}{8}}v_{xx}}  \lesssim \norm[X]{v}.
\end{align*}
Finally Lemma \ref{lem:infinitydecay} implies that
\begin{align*}
\norm[L^\infty]{(1+x^2)^{\frac{1}{32}} v_{xx}} \lesssim \norm[L^{2}]{(1+x^{2})^{\frac{1}{16}}v_{xx}} + \norm[L^2]{v_{xxx}} \lesssim \norm[X]{v}
\end{align*}
		
\item For $\Hil v$, we have $\norm[\dot{H}^{s}]{\Hil v} = \norm[\dot{H}^{s}]{v} $ for $s \geq 0$ and using \cref{prop:JdeltaHilnorm} we see that $\norm[L^{2}_{w}]{\Hil v}  \lesssim \norm[L^{2}_{w}]{v}$ and
\begin{align*}
\norm[L^{2}_{w}]{x \Hil v_{x}} =\norm[L^{2}_{w}]{\Hil \brac{x v_{x}}}\lesssim \norm[L^{2}_{w}]{x v_{x}}.
\end{align*}
since $\sqbrac{\Hil,x} v_{x} = 0$ due to the decay of $v$ at $\pm \infty$. Hence, the properties proved for $v$ apply to $\Hil v$.
\end{enumerate}
\end{proof}

\begin{cor}\label{cor:furtherproperties}
Let $\ep_{1} > 0$ be given by \cref{prop:weightproperties} and let $0 \leq \abs{\ep} \leq \ep_1$. 
Let $u\in L^{2}_{w}(\mathbb{R})\cap \dot{H}^{1}(\mathbb{R}) \cap \dot{H}^{3}(\mathbb{R})$ and $xu_{x}\in L^{2}_{w}(\mathbb{R})$ and let $g = \ep G + u$ and $f =  i\Hil g$. Then
\begin{enumerate}
\item We have
\begin{align*}
&  \norm[L^2_{w}]{ (1+x^2)^\half g_x}  + \norm[L^{\infty}]{\brac{1+x^{2}}^{\frac{1}{8}}g_{x}}   + \norm[L^{2}]{(1+x^{2})^{\frac{1}{16}}g_{xx}} \\
&  + \norm[L^{\infty}]{(1+x^{2})^{\frac{1}{32}}g_{xx}} + \norm[L^{2}]{g_{xxx}} \\
&  \lesssim \abs{\ep} + \norm[X]{u}.
\end{align*}

\item We have
\begin{align*}
&  \norm[L^2]{ (1+x^2)^{\frac{1}{8}} f_x}  + \norm[L^{\infty}]{\brac{1+x^{2}}^{\frac{1}{8}} f_{x}}   + \norm[L^{2}]{(1+x^{2})^{\frac{1}{16}} f_{xx}} \\
&  + \norm[L^{\infty}]{(1+x^{2})^{\frac{1}{32}} f_{xx}} + \norm[L^{2}]{f_{xxx}} \\
&  \lesssim \abs{\ep} + \norm[X]{u}.
\end{align*}
\end{enumerate}
\end{cor}
\begin{proof}
This follows directly from \cref{prop:furtherproperties} and \cref{prop:Gproperties}. 
\end{proof}

\section{The Full Nonlinear Equation}\label{sec:fullequation}

In this section, we will prove the main result \cref{thm:mainselfsimilarresult} by first proving the existence and uniqueness of a smooth solution to the equation \cref{eq:linearequationorig} with $F$ is given by the nonlinear expression \eqref{F}.  From \eqref{F} we see that we can write $F$ as
\begin{align}\label{eq:Fnonlineardef}
F = \ep S + N
\end{align}
where
\begin{align}\label{def:S}
S = - i\Hil G_{xxx} -  w\frac{x}{3a}G_x
\end{align}
and
\begin{align}\label{def:N}
\begin{aligned}
N & =   ie^f \sqbrac{e^{-f}, \Hil}g_{xxx}  - e^f i\Hil\cbrac{e^{-f}f_x^2 g_x - e^{-f}f_{xx}g_x -2e^{-f}f_xg_{xx}}  \\
& \quad - w(e^{3i\Hil u} - 1)\frac{x}{3a}g_x  + ie^f f_{x} \Hil\frac{d}{dx} \cbrac{e^{-f}g_{x}} + e^{3f}g_{x}\Hil\cbrac{e^{-2f}\Hil\frac{d}{dx}\cbrac{ e^{-f}g_{x}}} 
\end{aligned}
\end{align}

We have the following estimate for the expression \eqref{def:N} in terms of the norm given by \cref{def:Xnorm}.
\begin{prop}\label{prop:contraction}
Let $\ep_1>0$ be as given by \cref{prop:weightproperties} and let $0 \leq \abs{\ep} \leq \ep_1$. Let $u, \util \in L^2_w(\Rsp)$ be such that $\norm[X]{u}, \norm[X]{\util} \leq 1$. Define $g = \ep G + u$ and $f =  i\Hil g$ and let $N$ be given by \cref{def:N}. Further, define $\gtil, \ftil$ and $\Ntil$ analogously with $u$ replaced by $\util$. 

Then, for $b= 0$ or $b=-\frac{1}{2}$, there exists a continuous strictly increasing function $C_{b} : [0,\infty)\rightarrow [0,\infty)$ with $C_{b}(0) = 0$ such that
\begin{align*}
\norm[L^{2}]{w^{b}(N - \Ntil)} &\leq C_{b}(\abs{\ep} + \norm[X]{u}+\norm[X]{\util}) \norm[X]{u- \util}
\end{align*}
\end{prop}
\begin{proof}
Observe that
\begin{align*}
e^f = e^{i\ep\Hil G}e^{i\Hil u} = w^{\frac{1}{3}}e^{i\Hil u}
\end{align*}
From \cref{prop:furtherproperties} we see that $\norm[L^\infty]{{i\Hil u}} + \norm[L^\infty]{{i\Hil \util}} \lesssim 1$ and hence
\begin{align*}
\norm[L^2_w]{e^{i\Hil u} - e^{i\Hil \util}} + \norm[L^\infty]{(1+x^2)^{\frac{1}{8}}(e^{i\Hil u} - e^{i\Hil \util})} \lesssim \norm[X]{u - \util}
\end{align*}
Let $b = 0$ or $b = -1/2$. We will use the notation $C_{b}$ to represent any continuous strictly increasing function $C_{b} : [0,\infty)\rightarrow [0,\infty)$ with $C_{b}(0) = 0$. Now
\begin{align*}
w^b e^f \sqbrac{e^{-f}, \Hil}g_{xxx} & = [w^b, \Hil]g_{xxx} - [w^be^f, \Hil](e^{-f}g_{xxx}) \\
& = [w^b, \Hil]g_{xxx} - [w^be^f, \Hil](e^{-f}g_{xx})_x -[w^be^f,\Hil](e^{-f} f_x g_{xx})
\end{align*}
From \cref{prop:standardcommutator} we get
\begin{align*}
\norm[L^2]{[w^b, \Hil]g_{xxx}}  & \lesssim \norm[L^\infty]{(w^b)_x}\norm[L^2]{g_{xx}}   \\
\norm[L^2]{[w^be^f, \Hil](e^{-f}g_{xx})_x}  & \lesssim \norm[L^\infty]{(w^be^f)_x}\norm[L^2]{e^{-f}g_{xx}}  \\
\norm[L^2]{[w^be^f,\Hil](e^{-f} f_x g_{xx})} & \lesssim  \norm[L^2]{(w^b e^f)_x} \norm[L^2]{e^{-f}f_x} \norm[L^2]{g_{xx}}
\end{align*}
Now subtracting terms individually and using \cref{prop:furtherproperties}, \cref{cor:furtherproperties} and \cref{prop:standardcommutator} we therefore get
\begin{align*}
\norm[L^2]{w^b e^f \sqbrac{e^{-f}, \Hil}g_{xxx} - w^b e^{\ftil} \sqbrac{e^{-\ftil}, \Hil}\gtil_{xxx}} \leq C_{b}(\abs{\ep} + \norm[X]{u}+\norm[X]{\util}) \norm[X]{u- \util}
\end{align*}
The rest of the terms of $N$ are also controlled easily by again subtracting terms individually and using \cref{prop:furtherproperties}, \cref{cor:furtherproperties} and \cref{prop:JdeltaHilnorm}. 
\end{proof}

We can now prove the existence and uniqueness of  solutions to \eqref{eq:linearequationorig}. 

\begin{thm}\label{thm:mainexistencesmalldata}
There exists $\ep_0 > 0$ such that for all $\ep \in [-\ep_0, \ep_0]$, there exists a $u \in L^{2}_{w}(\mathbb{R})\cap \dot{H}^{1}(\mathbb{R})\cap \dot{H}^{3}(\mathbb{R})$ with $xu_{x} \in L^{2}_{w}(\mathbb{R})$ and $\norm[X]{u} \leq \sqrt{\ep_0}$, which solves \eqref{eq:linearequationorig}. Moreover if $u_1, u_2$ are two such solutions to \eqref{eq:linearequationorig} with $\norm[X]{u_1} \leq \sqrt{\ep_0}$ and $\norm[X]{u_2} \leq \sqrt{\ep_0}$, then $u_1 = u_2$. 
\end{thm}
\begin{proof}
Let $0 < \ep_1 \leq 1/100$ be as given by \cref{prop:weightproperties}. In \cref{thm:linear} let $\widetilde{C}_1>0$ be the constant in the inequality \cref{ineq:linearXestimate} i.e. if $v$ solves $\Lcal v = F$ then we get
\begin{align}\label{def:C1til}
\norm[X]{v} &\leq \widetilde{C}_1 \brac{\norm[L^{2}]{F} + \norm[L^{2}]{w^{-\frac{1}{2}}F}}
\end{align}
Recall that $F = \ep S + N$ where $S$ and $N$ are given by \cref{def:S} and \cref{def:N}. By Proposition \ref{prop:Gproperties}, we have that $S, w^{-\frac{1}{2}}S\in L^{2}(\mathbb{R})$. Let $\widetilde{C}_2>0$ be such that
\begin{align}\label{def:C2til}
\norm[L^2]{S} + \norm[L^2]{w^{-\frac{1}{2}}S} \leq \widetilde{C}_2
\end{align}

As $C_b(0) = 0$ for both $b = 0, - \half$ and using the fact that $C_b$ are continuous strictly increasing functions, we see that there exists $0 < \ep_0 \leq \ep_1$ such that 
\begin{align}\label{ep0C0Cminushalfcondition}
\widetilde{C}_1\cbrac{C_0(\ep_0 + 2\sqrt{\ep_0}) + C_{-\half}(\ep_0 + 2\sqrt{\ep_0})}  \leq \frac{1}{4}
\end{align}
and
\begin{align}\label{conditionep0cic2}
\sqrt{\ep_0}\widetilde{C_1}\widetilde{C_2} \leq \frac{1}{4}
\end{align}
We use this $\ep_0>0$ and let $0 \leq \abs{\ep} \leq \ep_0$.  

We let $N_0 = 0$ and $F_0 = \ep S$. From  \cref{prop:Gproperties} and  \cref{prop:weightproperties} we see that $F_0, w^{-\half} F_0 \in L^2(\Rsp)$. Hence by Theorem \ref{thm:linear}, we can define $u_{1}\in L^{2}_{w}(\mathbb{R}) \cap \dot{H}^{1}(\mathbb{R})\cap \dot{H}^{3}(\mathbb{R})$ with $x\partial_{x}u_{1} \in L^{2}_{w}(\mathbb{R})$ as the unique solution to
\begin{align*}
	\mathcal{L}u_{1} = F_0
\end{align*}
For $n \geq 1$ we define $g_{n} = u_{n} + \ep G$ and $f_{n} = i\Hil g_{n}$. We further define $F_n$ and $N_n$ by \cref{eq:Fnonlineardef} and \cref{def:N} respectively, where we use $u_n, g_n$ and $f_n$ in the formula of $N$ in \cref{def:N}. From \cref{prop:contraction} we see that if $u_{n}\in L^{2}_{w}(\mathbb{R}) \cap \dot{H}^{1}(\mathbb{R})\cap \dot{H}^{3}(\mathbb{R})$ with $x\partial_{x}u_{n} \in L^{2}_{w}(\mathbb{R})$, then we have $F_n, w^{-\half} F_n \in L^2(\Rsp)$. Hence using \cref{thm:linear}, we define $u_{n+1} \in L^{2}_{w}(\mathbb{R}) \cap \dot{H}^{1}(\mathbb{R})\cap \dot{H}^{3}(\mathbb{R})$ with $x\partial_{x}u_{n+1} \in L^{2}_{w}(\mathbb{R})$ as the unique solution to the equation
\begin{align*}
\mathcal{L}u_{n+1} = F_n
\end{align*}



As $\Lcal u_1 = F_0 =  \ep S$ we see from \cref{def:C1til}, \cref{def:C2til} and \cref{conditionep0cic2} that
\begin{align*}
\norm[X]{u_1} \leq \abs{\ep}\widetilde{C}_1\widetilde{C}_2 \leq \frac{\sqrt{\ep_0}}{4}
\end{align*}
Also as $\Lcal u_2 = F_1$ we see that
\begin{align*}
\Lcal(u_2 - u_1) = N_1 - N_0 = N_1
\end{align*}
and so from \cref{thm:linear}, \cref{prop:contraction}, \cref{def:C1til} and \cref{ep0C0Cminushalfcondition} we see that
\begin{align*}
\norm[X]{u_2 - u_1} & \leq \widetilde{C}_1\brac{\norm[L^2]{N_1} + \norm[L^2]{w^{-\half}N_1}} \\
& \leq \widetilde{C}_1\brac{C_0(\abs{\ep} + \norm[X]{u_{1}}) + C_{-\half}(\abs{\ep} + \norm[X]{u_{1}})}\norm[X]{u_1}  \\
& \leq \frac{\norm[X]{u_1}}{4} \\
& \leq \frac{\sqrt{\ep_0}}{16}
\end{align*}
Now observe that for $n \geq 2$ we have
\begin{align*}
\mathcal{L}\brac{u_{n+1}-u_{n}} &= N_{n}- N_{n-1}.
\end{align*}
Hence by \cref{thm:linear}, \cref{prop:contraction} and \cref{def:C1til} we get
\begin{align*}
& \norm[X]{u_{n+1}-u_{n}} \\
& \leq \widetilde{C}_1 \brac{ \norm[L^{2}]{N_{n}- N_{n-1}} + \norm[L^{2}]{w^{-\frac{1}{2}}\brac{N_{n}- N_{n-1}}}} \\
& \leq \widetilde{C}_1\brac{C_0(\ep + \norm[X]{u_{n}} + \norm[X]{u_{n-1}}) + C_{-\half}(\ep + \norm[X]{u_{n}} + \norm[X]{u_{n-1}})}\norm[X]{u_n - u_{n-1}} 
\end{align*}
Hence using \cref{ep0C0Cminushalfcondition} and by induction we see that for all $n \geq 1$ we have $\norm[X]{u_n} \leq \sqrt{\ep_0}$ and for all $n \geq 2$ we have
\begin{align}\label{ineq:contraction}
\norm[X]{u_{n+1}-u_{n}} &\leq \frac{1}{4} \norm[X]{u_{n} - u_{n-1}}.
\end{align}
Thus, $u_{n}$ converges to some $u \in L^{2}_{w}(\mathbb{R})\cap \dot{H}^{1}(\mathbb{R})\cap \dot{H}^{3}(\mathbb{R})$ with $xu_{x}\in L^{2}_{w}(\mathbb{R})$ in norm i.e.
\begin{align*}
\lim_{n\rightarrow \infty}\norm[X]{u_{n} - u} = 0.
\end{align*}
Moreover we have $\norm[X]{u} \leq \sqrt{\ep_0}$. 
Thus, by Proposition \ref{prop:contraction}, it can be seen that $u$ satisfies \eqref{eq:linearequationorig} in a distributional sense. Since  $u \in L^{2}_{w}(\mathbb{R})\cap \dot{H}^{1}(\mathbb{R})\cap \dot{H}^{3}(\mathbb{R})$ and $xu_{x}\in L^{2}_{w}(\mathbb{R})$, it follows that $u$ satisfies  \eqref{eq:linearequationorig} pointwise a.e. The uniqueness statement follows immediately from \cref{prop:contraction}, \cref{def:C1til} and \cref{ep0C0Cminushalfcondition}. 
\end{proof}


Now consider the solution $u\in L^{2}_{w}(\mathbb{R})\cap \dot{H}^{1}(\mathbb{R})\cap \dot{H}^{3}(\mathbb{R})$ with $xu_{x}\in L^{2}_{w}(\mathbb{R})$ to \cref{eq:linearequationorig}  with $F$ given by \eqref{F}, as given by \cref{thm:mainexistencesmalldata}. Differentiating both sides of \cref{eq:linearequationorig} we get
\begin{align*}
\Lcal \util = \widetilde{F}
\end{align*}
where $\util = u_x \in H^2(\Rsp)$ and  
\begin{align*}
\widetilde{F} = -\brac{w\frac{x}{3a}}_{x}u_{x} + \partial_{x}F
\end{align*}
It is easy to see from the expression of \cref{F} and the proof of \cref{prop:contraction} that $\widetilde{F}, w^{-\half}\widetilde{F} \in L^2(\Rsp)$. Furthermore as $x\widetilde{u} \in L^2_w(\Rsp)$ we can therefore apply \cref{thm:linear} to conclude that $\norm[X]{\widetilde{u}} < \infty$. Hence  $u \in \dot{H}^{4}(\mathbb{R})$ and $xu_{xx}\in L^{2}_{w}(\mathbb{R})$. We can repeat this argument to conclude that $u \in \dot{H}^{k}(\mathbb{R})$ and $x\partial_{x}^{k}u \in L^{2}_{w}(\mathbb{R})$ for any $k\geq 1$. Therefore, we finally get

\begin{thm}\label{thm:fullregularity}
Let $u\in L^{2}_{w}(\mathbb{R})\cap \dot{H}^{1}(\mathbb{R}) \cap \dot{H}^{3}(\mathbb{R})$ with $xu_{x}\in L^{2}_{w}(\mathbb{R})$ be the unique solution to \eqref{eq:linearequationorig} with $\norm[X]{u} \leq \sqrt{\ep_0}$ as given by \cref{thm:mainexistencesmalldata}. Then $u$ is smooth and in particular $\norm[X]{\partial_x^k u} < \infty$ and hence $u\in \dot{H}^{k}(\mathbb{R})$ and $x\partial_{x}^{k}u \in L^{2}_{w}(\mathbb{R})$ for all $k\geq 1$.
\end{thm}

To complete the proof of \cref{thm:mainselfsimilarresult}, it remains to demonstrate that the solution $u$ of Theorem \ref{thm:fullregularity} 
leads to a solution of the equation \eqref{eq:dtZbar}. As we are dealing with self similar solutions, it is enough to demonstrate the existence of an $\eta$ which solves \eqref{eq:etaselfsimilar}. 

Let $u$ be as given by Theorem \ref{thm:fullregularity} and let  $g = u + \ep G$ and $f = i\Hil g$. Define 
\begin{align*}
\eta(x) = L + \int_0^x e^{(f+ig)(s)} \diff s
\end{align*}
where $L$ is a constant which we will specify later on. Hence $\eta_x = e^{f+ig}$. Observe that as $u$ is smooth, we have that $f, g$ are both smooth and hence $\eta$ is also smooth. Also the asymptotic properties required in \cref{thm:mainselfsimilarresult} now follow directly the properties of $f$ and $g$. 

As \cref{eq:selfsimilar} is equivalent to \cref{eq:linearequationorig}, we see that $f$ and $g$ satisfy \eqref{eq:selfsimilar} and therefore, \eqref{eq:xgxselfsimilar}. Now define the function $U$ as
\begin{align*}
U &=  \frac{1}{3}-\frac{1}{3a} - \frac{x}{3a}f_{x} + i\frac{x}{3a}g_{x} \\
& \quad +\frac{i}{\bar{\eta}_{x}}\frac{d}{dx}\brac{\bar{\eta}_{x}(I-\Hil)\cbrac{\Real\left(\frac{1}{\abs{\eta_{x}}^{2}}\frac{d}{dx}(\Id + \Hil)\cbrac{ \frac{1}{\eta_{x}}\frac{d}{dx} \left(\frac{\eta_{x}}{\abs{\eta_{x}}}\right)}\right)}}
\end{align*}
Observe that as \eqref{eq:xgxselfsimilar} holds, this implies that $\Imag U = 0$. Also observe that proving that $\eta_{x}$ satisfies \eqref{eq:18'} is equivalent to showing that $U = 0$ identically.

As $\norm[X]{\partial_x^k u} < \infty$ for all $k \geq 0$, we see from \cref{cor:furtherproperties} and \cref{prop:Gproperties} that for all $k \geq 1$ we have $\brac{1+x^{2}}^{\frac{1}{8}}\partial_{x}^{k} g(x) \in L^2(\Rsp)$, $\brac{1+x^{2}}^{\frac{1}{8}}\partial_{x}^{k}f(x) \in L^2(\Rsp)$, $\brac{1+x^{2}}^{\frac{1}{8}}\partial_{x}^{k}f(x) \in L^\infty(\Rsp)$ and $\brac{1+x^{2}}^{\frac{1}{8}}\partial_{x}^{k}g(x) \in L^\infty(\Rsp)$. Therefore by \cref{lem:infinitydecay} we see that the right hand side of \eqref{eq:xgxselfsimilar} decays faster than $\brac{1+x^{2}}^{-\frac{1}{16}}$. Hence, the left hand side term $xg_{x}$ has the same decay rate and so $g_{x}\in L^{1}(\mathbb{R})$. Applying the Hilbert transform to both sides of \cref{eq:xgxselfsimilar} and using the same logic, we also see that $i\Hil\brac{xg_{x}}$ is well defined and $\brac{1+x^{2}}^{\frac{1}{16}}i\Hil\brac{xg_{x}}\rightarrow 0$ as $x\rightarrow \pm \infty$. Now observe that
\begin{align*}
i\Hil\brac{xg_{x}} &= xf_{x} - \sqbrac{x,i\Hil}g_{x}
\end{align*}
where
\begin{align*}
\sqbrac{x,i\Hil}g_{x} &=  \frac{1}{\pi} \text{p.v.} \int_{\Rsp} g_{y}(y) \diff y = \frac{2\ep}{\pi}
\end{align*}
since $g\rightarrow \pm \ep$ and $x\rightarrow \pm \infty$. Thus, we have shown that $xf_{x}\rightarrow \frac{2\ep}{\pi}$ as $x \rightarrow \pm \infty$. Hence by \eqref{eq:a}, we see that
\begin{align*}
 \frac{1}{3}-\frac{1}{3a} - \frac{x}{3a}f_{x} + i\frac{x}{3a}g_{x} &=   \frac{1}{3a}\brac{\frac{2\ep}{\pi}-xf_{x}} + i\frac{x}{3a}g_{x}\rightarrow 0 \quad \text{as} \quad x\rightarrow \pm\infty.
\end{align*}
Using the decay of $\partial_{x}^{k}g$ and $\partial_{x}^{k}f$, the rest of $U$ is bounded and decays to zero as $x\rightarrow \pm \infty$. Thus, $U$ is bounded and decays to zero at infinity. From the above analysis, we also see that $\Hil U = -U$ and so $\Real(U) = -i\Hil \Imag(U)$. As $\Imag(U) = 0$, we have that $\Real(U) = 0$ and so $U = 0$ identically.  Thus, we have that $\eta_{x}$ satisfies \eqref{eq:18'}.

Finally if we define a function $S$ as  
\begin{align*}
S = \frac{x}{a}\bar{\eta}_{x}-3\bar{\eta}_{x}i(I-\Hil)\cbrac{\Real\left(\frac{1}{\abs{\eta_{x}}^{2}}\frac{d}{dx}(\Id + \Hil)\cbrac{ \frac{1}{\eta_{x}}\frac{d}{dx} \left(\frac{\eta_{x}}{\abs{\eta_{x}}}\right)}\right)}.
\end{align*}
then a direct computation using the equation \cref{eq:18'} yields that $S_x = \bar{\eta}_x$. Hence we can now specify $L$ in the definition of $\eta$ by directly defining $\eta = \bar{S}$. This completes the proof of \cref{thm:mainselfsimilarresult}.

\section{Appendix}

In this appendix, we state several useful propositions that are used throughout the paper.

\begin{prop}\label{prop:standardcommutator}
	Let $f,g \in \mathcal{S}(\Rsp)$ with $s,a\in \Rsp$ and $m,n \in \Zsp$. Then we have the following estimates
	\begin{enumerate}
		\item $\norm*[\big][2]{\abs{D}^s\sqbrac{f,\Hil}(\abs{D}^{a} g )} \lesssim_{s,a}  \norm*[\big][BMO]{\abs{D}^{s+a}f}\norm[2]{g}$ \quad  for $s,a \geq 0$
		
		\item $\norm*[\big][2]{\abs{D}^s\sqbrac{f,\Hil}(\abs{D}^{a} g )} \lesssim_{s,a}  \norm*[\big][2]{\abs{D}^{s+a}f}\norm[BMO]{g}$ \quad  for $s\geq 0$ and $a>0$
		
		\item $\norm*[\big][2]{\sqbrac*[\big]{f,\abs{D}^\half}g } \lesssim \norm*[\big][BMO]{\abs{D}^\half f}\norm[2]{g} \lesssim \norm*[\big][L^{2}]{ f_{x}}\norm[2]{g}$

		\item $\norm*[\big][2]{\sqbrac*[\big]{f,\abs{D}^\half}(\abs{D}^\half g) } \lesssim \norm*[\big][BMO]{\abs{D} f}\norm[2]{g}$
		
		\item $\norm[\Linfty\cap\Hhalf]{\partial_{x}^m\sqbrac{f,\Hil}\partial_{x}^n g} \lesssim_{m,n} \norm*[\big][2]{\partial_{x}^{(m+n+1)}f}\norm[2]{g}$ \quad  for $m,n \geq 0$
		
		\item $\norm[2]{\partial_{x}^m\sqbrac{f,\Hil}\partial_{x}^n g} \lesssim_{m,n} \norm*[\infty]{\partial_{x}^{(m+n)} f}\norm[2]{g}$ \quad  for $m,n \geq 0$
		
		\item $\norm[2]{\partial_{x}^m\sqbrac{f,\Hil}\partial_{x}^n g} \lesssim_{m,n} \norm*[2]{\partial_{x}^{(m+n)} f}\norm[\infty]{g}$ \quad for $m\geq 0$ and $n\geq 1$
		
		\item $\norm[2]{\sqbrac{f,\Hil}g} \lesssim \norm[2]{f_{x}}\norm[1]{g}$
	\end{enumerate}
\end{prop}
\begin{proof}
	The above estimates are shown in the Appendix of \cite{Ag21}.
\end{proof}
Next, let $J_{\delta}$ be defined as in \eqref{def:Jdelta}.
\begin{prop}\label{prop:appendixJdelta}
	Let $f,g \in \mathcal{S}(\Rsp)$ and let $0 < \del \leq 1$. Then we have
	\begin{enumerate}
		\item $\norm[2]{\sqbrac{\Jdel,f}g_{x}} \lesssim \del\norm[H^2]{f}\norm[2]{g_{x}}$ 
		\item $\norm[2]{\sqbrac{\Jdel,f} g_{x}} \lesssim \norm[\infty]{f_{x}}\norm[2]{g}$ 
		\item $\norm[2]{\abs{D}^\half\sqbrac{\Jdel,f}g_{x}} \lesssim \norm[\infty]{f_{x}}\norm[\Hhalf]{g}$
	\end{enumerate}
\end{prop}
\begin{proof}
The above estimates are proven in Section 7 of \cite{Ag21}.
\end{proof}

\begin{prop}\label{prop:JdeltaHilnorm}
Fix $-\frac{1}{4} \leq \alpha \leq \frac{1}{4}$  and let the weight $W:\mathbb{R}\rightarrow (0,\infty)$ satisfy
\begin{align}\label{eq:weightcondition}
	(1 + \abs{x})^{\alpha} \lesssim W(x)  \lesssim (1 + \abs{x})^{\alpha} \qq \tx{ for all } x \in \Rsp
\end{align}
Let $f: \Rsp \to \Rsp$ and let $M$ is the Hardy Littlewood maximal operator. If $Wf\in L^{2}(\mathbb{R})$ and $J_{\delta}$ is the operator defined in \eqref{def:Jdelta}, then for all $0 < \delta \leq 1$,
\begin{align*}
\norm[L^{2}]{WJ_{\delta}f} + \norm[L^{2}]{W M(f)} + \norm[L^{2}]{W\Hil f}&\lesssim \norm[L^{2}]{Wf}.
\end{align*}
\end{prop}
\begin{proof}
It is easy to check that \eqref{eq:weightcondition} implies that $W^2$ is a Muckenhoupt $A_2$ weight. Hence as the Hilbert transform is bounded on weighted $L^2$ spaces, we immediately get $ \norm[L^{2}]{W\Hil f} \lesssim \norm[L^{2}]{Wf}$. We also see that 
\begin{align*}
\abs{W(x)J_\delta f(x)} \lesssim W(x)M(f)(x) \qq \tx{ for all } x \in \Rsp
\end{align*}
As the maximal operator is also bounded on weighted $L^2$ spaces, we therefore see that $\norm[L^{2}]{WJ_{\delta}f} \lesssim \norm[L^{2}]{W M(f)} \lesssim \norm[L^{2}]{Wf}$. Hence proved. 
\end{proof}

\begin{prop}\label{prop:commutator}
Consider the operator $J_{\delta}$ as defined in \eqref{def:Jdelta} for $0<\delta \leq 1$. Consider $\ep, \ep', \ep''\geq 0$.	Suppose for all $x\neq y$, the function  $W:\mathbb{R}\rightarrow\mathbb{R}$ satisfies \begin{align*}
		\abs{\frac{W(y)-W(x)}{y-x}} \lesssim 1+ \abs{x}^{\ep} + \abs{y}^{\ep}.
	\end{align*}
Further, assume that we have weights $\nu, \mu:\mathbb{R}\rightarrow (0,\infty)$ so that for all $x\in\mathbb{R}$,
\begin{align*}
\abs{\nu(x)}, \abs{\nu(x)}^{-1}\lesssim 1+\abs{x}^{\ep'}
\end{align*} and 
\begin{align*}
\abs{\mu(x)}, \abs{\mu(x)}^{-1}\lesssim 1+\abs{x}^{\ep''}.
\end{align*}
If $q\in L^{2}(\mathbb{R})$ is compactly supported in the interval $[-\frac{1}{\delta_{0}},\frac{1}{\delta_{0}}]$ for some $0 < \delta_{0} \leq 1$, then
	\begin{align*}
		\norm[L^{2}]{\nu \sqbrac{J_{\delta},W}q} \lesssim \delta \delta_{0}^{-\ep-\ep'-\ep''}\norm[L^{2}]{\mu q}.
	\end{align*}
\end{prop}
\begin{proof}
	We have the pointwise estimate:
\begin{align*}
	\abs{\sqbrac{J_{\delta},W}q(x)}&= \abs{J_{\delta}\brac{W q}(x) - W(x) J_{\delta}q(x)}\\
	&= \abs{\int_{\mathbb{R}} \frac{1}{\delta}\phi\brac{\frac{x-y}{\delta}}(W(y)-W(x)) q(y) \diff y}\\
	&\leq \delta \int_{\mathbb{R}} \frac{1}{\delta}\phi\brac{\frac{x-y}{\delta}}\frac{\abs{x-y}}{\delta}\abs{\frac{W(y)-W(x)}{y-x}} \abs{q(y)} \diff y\\
	&\lesssim \delta \int_{\mathbb{R}} \frac{1}{\delta}\phi\brac{\frac{x-y}{\delta}}\frac{\abs{x-y}}{\delta}\brac{1 + \abs{x}^{\ep} + \abs{y}^{\ep}} \abs{q(y)} \diff y.
\end{align*}	
We have that $\phi\brac{\frac{x-y}{\delta}} = 0$ if $\abs{x-y} > \delta$ and $q(y) = 0$ if $\abs{y} > \frac{1}{\delta_{0}}$. Hence,
	\begin{align*}
	\abs{\sqbrac{J_{\delta},W}q(x)}&\lesssim \delta \delta_{0}^{-\ep}\int_{\mathbb{R}} \frac{1}{\delta}\phi\brac{\frac{x-y}{\delta}}\frac{\abs{x-y}}{\delta} \abs{q(y)}  \diff y\\
	&\lesssim \delta \delta_{0}^{-\ep-\ep''}\int_{\mathbb{R}} \frac{1}{\delta}\phi\brac{\frac{x-y}{\delta}}\frac{\abs{x-y}}{\delta} \abs{\mu(y)}\abs{q(y)}  \diff y.
\end{align*}
Since $\sqbrac{J_{\delta},W}q(x)$ is supported in the region $\abs{x}\leq \frac{2}{\delta_{0}}$, we obtain that
	\begin{align*}
		\abs{\nu(x)\sqbrac{J_{\delta},W}q(x)} &\lesssim \delta_{0}^{-\ep'}\abs{\sqbrac{J_{\delta},W}q(x)}.
	\end{align*}
Thus, by Young's inequality for convolutions,
	\begin{align*}
		\norm[L^{2}]{\nu \sqbrac{J_{\delta},W}q} \lesssim \delta \delta_{0}^{-\ep-\ep'-\ep''}\norm[L^{2}]{\mu q}.
	\end{align*}
	This concludes the proof.
\end{proof}

\begin{prop}\label{prop:deltaconvergenceL2}
Let $0 \leq \ep \leq \frac{1}{10}$. Suppose  $W:\mathbb{R}\rightarrow\mathbb{R}$ satisfies \begin{align*}
	\abs{\frac{W(y)-W(x)}{y-x}} \lesssim 1 + \abs{x}^{\ep} + \abs{y}^{\ep}
\end{align*}
 for all $x\neq y$ and
 \begin{align*}
 \abs{W(x)} \lesssim 1+ \abs{x}^{1+\ep}
 \end{align*}
 for all $x\in \mathbb{R}$. Let $0< \delta \leq 1$ and consider the operator $J_{\delta}$ defined in \eqref{def:Jdelta} and the smooth cutoff function $\chi_{\delta}$ defined in \eqref{def:chidelta}. Let $f\in L^{2}_{loc}\brac{\mathbb{R}}$ and assume that there exists a uniform constant $C_{0}>0$ such that for all $0 < \delta \leq 1$ we have
\begin{align*}
\norm[L^{2}]{W \chi_{\delta}J_{\delta}f} \leq C_{0}
\end{align*}
Then, $W f \in L^{2}(\mathbb{R})$ and moreover,
\begin{align*}
\norm[L^{2}]{W f} \leq C_{0}.
\end{align*}
\end{prop}
\begin{proof}
Consider a fixed $0 < \tau \leq 1$, consider $\delta$ such that $0 < \delta < \frac{\tau}{2}$, and set $\tau' = \frac{\tau}{4}$. Then,
	\begin{align*}
	\chi_{\tau}J_{\delta}\brac{\chi_{\tau'}f} = \chi_{\tau}J_{\delta}f = \chi_{\tau} \chi_{\delta}J_{\delta}f.
	\end{align*}
Hence, we have that
\begin{align}\label{ineq:chitaubound}
	\norm[L^{2}]{W \chi_{\tau}J_{\delta}\brac{\chi_{\tau'}f}} = \norm[L^{2}]{W \chi_{\tau} \chi_{\delta}J_{\delta}f} \leq  C_{0}.
\end{align}
Next, by Proposition \ref{prop:commutator},
\begin{align*}
\norm[L^{2}]{W\chi_{\tau}J_{\delta}\brac{\chi_{\tau'}f} - W\chi_{\tau}f} &\leq \norm[L^{2}]{\sqbrac{W\chi_{\tau},J_{\delta}}\brac{\chi_{\tau'}f}}  + \norm[L^{2}]{J_{\delta}\brac{W\chi_{\tau}\chi_{\tau'}f} -W\chi_{\tau}f}\\
&\lesssim \delta \tau^{-2}\norm[L^{2}]{W \chi_{\tau'}f} + \norm[L^{2}]{J_{\delta}\brac{W\chi_{\tau}f} -W\chi_{\tau}f}.
\end{align*}
Since $W \chi_{\tau}f \in L^{2}$ for fixed $0 <\tau \leq 1$, we know that
\begin{align*}
\lim_{\delta\rightarrow 0}\norm[L^{2}]{J_{\delta}\brac{W\chi_{\tau}f} -W\chi_{\tau}f} = 0.
\end{align*}
Thus, for fixed $ 0 < \tau \leq 1$,
\begin{align*}
\lim_{\delta\rightarrow 0}\norm[L^{2}]{W\chi_{\tau}J_{\delta}\brac{\chi_{\tau'}f} - W\chi_{\tau}f}  = 0.
\end{align*}
By \eqref{ineq:chitaubound}, it follows that $\norm[L^{2}]{W \chi_{\tau}f} \leq C_{0}$ for all $0 < \tau \leq 1$. Thus, we can conclude that $\norm[L^{2}]{W f} \leq C_{0}$.
\end{proof}

\begin{prop}\label{prop:weightsobolevinterpolation}
Suppose $W:\mathbb{R}\rightarrow (0,\infty)$ satisfies $W \in \dot{H}^1(\Rsp)$ and that $W^{k}W_{x}\in L^{2}(\mathbb{R})$ for all $-10\leq k \leq 10$, $k \in \Zsp$. Also assume that either $W\in L^{\infty}(\mathbb{R})$ or $W^{-1}\in L^{\infty}(\mathbb{R})$. 
	
Let $f:\mathbb{R}\rightarrow \mathbb{R}$. If $W^3f \in \Ltwo(\mathbb{R})$ and $f \in \dot{H}^{\frac{3}{2}}(\mathbb{R})$, then $W^{2}\abs{D}^{\frac{1}{2}} f \in L^{2}(\mathbb{R})$ and $W \abs{D} f \in L^{2}(\mathbb{R})$. Moreover,
	\begin{align*}
		\norm[2]{W^2\abs{D}^{\frac{1}{2}} f} +	\norm[2]{W \abs{D} f} &\lesssim_{W}	\norm[2]{W^3 f} + \norm[\dot{H}^{\frac{3}{2}}]{f}.
	\end{align*}
\end{prop}
\begin{proof}
We will prove this theorem for $f\in\Scal(\mathbb{R})$. The theorem then follows by approximation. 

Let us first consider the case where $W^{-1}\in L^{\infty}(\mathbb{R})$. We have
	\begin{align*}\label{ineq:firstinequalityinterpolation}
		\norm[2]{W^2 \abs{D}^{\frac{1}{2}} f}^2 & = \int_\Rsp W^4 (\abs{D}^{\frac{1}{2}} f)(\abs{D}^{\frac{1}{2}} f) \diff x \\
		& = \int_\Rsp \brac{\sqbrac{\abs{D}^{\frac{1}{2}}, W^4}\abs{D}^{\frac{1}{2}} f}(f) \diff x + \int_\Rsp (W\abs{D} f)(W^3 f) \diff x\\
		& = \int_\Rsp \brac{W^{-3}\sqbrac{\abs{D}^{\frac{1}{2}}, W^4}\abs{D}^{\frac{1}{2}} f}(W^3f) \diff x + \int_\Rsp (W\abs{D} f)(W^3 f) \diff x.
	\end{align*}
	Then,
	\begin{align*}
		\norm[2]{W^{-3}\sqbrac{\abs{D}^{\frac{1}{2}}, W^4}\abs{D}^{\frac{1}{2}} f} & \lesssim \norm[\infty]{W^{-1}}^3\norm[2]{W^3 W_{x}}\norm[2]{\abs{D}^{\frac{1}{2}} f} \\
		& \lesssim \norm[\infty]{W^{-1}}^5\norm[2]{W^3 W_{x}}\norm[2]{W^2\abs{D}^{\frac{1}{2}} f}.
	\end{align*}
	Therefore,
	\begin{align*}
		\norm[2]{W^2 \abs{D}^{\frac{1}{2}} f}^2 \lesssim \norm[\infty]{W^{-1}}^5\norm[2]{W^3 W_{x}}\norm[2]{W^2 \abs{D}^{\frac{1}{2}} f}\norm[2]{W^3f} + \norm[2]{W\abs{D} f}\norm[2]{W^3 f}
	\end{align*}
	Hence
	\begin{align}
		\norm[2]{W^2 \abs{D}^{\frac{1}{2}} f}^2 \lesssim_{W} \norm[2]{W^3 f}^2 + \norm[2]{W\abs{D} f}\norm[2]{W^3 f}
	\end{align}
	Next, we have that
	\begin{align*}
		\norm[2]{W \abs{D} f}^2 & = \int_\Rsp W^2 (\abs{D} f)(\abs{D} f) \diff x \\
		& =  \int_\Rsp W^{-1}\brac{\sqbrac{ W^2, \abs{D}^{\frac{1}{2}}} \abs{D}^{\frac{1}{2}} f} W\abs{D} f \diff x  +  \int_\Rsp (W^2 \abs{D}^{\frac{1}{2}} f)(\abs{D}^{\frac{3}{2}} f) \diff x.
	\end{align*}
	We have that
	\begin{align*}
		W^{-1}\sqbrac{ W^2, \abs{D}^{\frac{1}{2}}} \abs{D}^{\frac{1}{2}} f &= \sqbrac{ W, \abs{D}^{\frac{1}{2}}} \abs{D}^{\frac{1}{2}} f - \sqbrac{ W^{-1}, \abs{D}^{\frac{1}{2}}} \brac{W^{2}\abs{D}^{\frac{1}{2}} f }
	\end{align*}
	where
	\begin{align*}
		\norm[L^{2}]{ \sqbrac{ W, \abs{D}^{\frac{1}{2}}} \abs{D}^{\frac{1}{2}} f} &\lesssim \norm[L^{2}]{W_{x}}\norm[L^{\infty}]{W^{-1}}^{2}\norm[L^{2}]{W^{2}\abs{D}^{\frac{1}{2}} f}
	\end{align*}
	and
	\begin{align*}
		\norm[L^{2}]{\sqbrac{ W^{-1}, \abs{D}^{\frac{1}{2}}} W^{2}\abs{D}^{\frac{1}{2}} f } &\lesssim \norm[L^{2}]{W^{-2}W_{x}}\norm[L^{2}]{W^{2}\abs{D}^{\frac{1}{2}} f}.
	\end{align*}
	Hence,
	\begin{align*}
		\norm[2]{W \abs{D} f}^2 & \lesssim\norm[L^{2}]{W_{x}}\norm[L^{\infty}]{W^{-1}}^{2}\norm[L^{2}]{W^{2}\abs{D}^{\frac{1}{2}} f}	\norm[2]{W \abs{D} f} + \norm[L^{2}]{W^{2}\abs{D}^{\frac{1}{2}} f}\norm[\dot{H}^{\frac{3}{2}}]{f}
	\end{align*}
	and thus,
		\begin{align*}
		\norm[2]{W \abs{D} f}^2 & \lesssim\norm[L^{2}]{W_{x}}^{2}\norm[L^{\infty}]{W^{-1}}^{4}\norm[L^{2}]{W^{2}\abs{D}^{\frac{1}{2}} f}^{2}	 + \norm[L^{2}]{W^{2}\abs{D}^{\frac{1}{2}} f}\norm[\dot{H}^{\frac{3}{2}}]{f}.
	\end{align*}
	In particular,
			\begin{align}\label{ineq:secondinterpolationinequality}
		\norm[2]{W \abs{D} f} & \lesssim_{W} \norm[L^{2}]{W^{2}\abs{D}^{\frac{1}{2}} f}	 + \norm[\dot{H}^{\frac{3}{2}}]{f}.
	\end{align}
	Plugging \eqref{ineq:secondinterpolationinequality} into \eqref{ineq:firstinequalityinterpolation}, we conclude the proof in the case where $W^{-1}\in L^{\infty}(\mathbb{R})$.
	
Now let us consider the case where $W\in L^{\infty}(\mathbb{R})$. As before, we have
	\begin{align*}
	\norm[2]{W^2 \abs{D}^{\frac{1}{2}} f}^2 & =  \int_\Rsp \brac{W^{-3}\sqbrac{\abs{D}^{\frac{1}{2}}, W^4}\abs{D}^{\frac{1}{2}} f}(W^3f) \diff x + \int_\Rsp (W\abs{D} f)(W^3 f) \diff x.
\end{align*}
In this case, we have that
	\begin{align*}
		& W^{-3}\sqbrac{\abs{D}^\half, W^4}\abs{D}^\half f \\
		& = \sqbrac{\abs{D}^\half, W}\abs{D}^\half f - \sqbrac{\abs{D}^\half, W^{-3}}(W^4\abs{D}^\half f) \\
		& =  \sqbrac{\abs{D}^\half, W}(W^{-2}W^2\abs{D}^\half f) - \sqbrac{\abs{D}^\half, W^{-3}}(W^4\abs{D}^\half f) \\
		& =  \sqbrac{\abs{D}^\half, W^{-1}}(W^2\abs{D}^\half f) -  W\sqbrac{\abs{D}^\half, W^{-2}}(W^2\abs{D}^\half f) - \sqbrac{\abs{D}^\half, W^{-3}}(W^4\abs{D}^\half f)
	\end{align*}
	These terms are bounded as follows:
	\begin{align*}
		\norm[L^{2}]{ \sqbrac{\abs{D}^\half, W^{-1}}(W^2\abs{D}^\half f)} &\lesssim \norm[L^{2}]{W^{-2}W_{x}}\norm[L^{2}]{W^2\abs{D}^\half f}
	\end{align*}
	and
	\begin{align*}
		\norm[L^{2}]{W\sqbrac{\abs{D}^\half, W^{-2}}(W^2\abs{D}^\half f)} &\lesssim  \norm[L^{\infty}]{W}\norm[L^{2}]{W^{-3}W_{x}}\norm[L^{2}]{W^2\abs{D}^\half f}
	\end{align*}
	and
	\begin{align*}
		\norm[L^{2}]{\sqbrac{\abs{D}^\half, W^{-3}}(W^4\abs{D}^\half f)} &\lesssim \norm[L^{\infty}]{W}^{2}\norm[L^{2}]{W^{-4}W_{x}}\norm[L^{2}]{W^2\abs{D}^\half f}
	\end{align*}
	Therefore
	\begin{align*}
		\norm[2]{W^2 \abs{D}^\half f}^2 \lesssim_{W} \norm[L^{2}]{W^2\abs{D}^\half f}\norm[2]{W^3 f} + \norm[2]{W\abs{D} f}\norm[2]{W^3 f}
	\end{align*}
	Hence
	\begin{align}\label{ineq:firstwinfinitypapabshalf}
		\norm[2]{W^2 \abs{D}^\half f}^2 \lesssim_{W} \norm[2]{W^3 f}^2 + \norm[2]{W\abs{D} f}\norm[2]{W^3 f}
	\end{align}
	Next, as before we have
	\begin{align*}
	\norm[2]{W \abs{D} f}^2
	& =  \int_\Rsp W^{-1}\brac{\sqbrac{ W^2, \abs{D}^{\frac{1}{2}}} \abs{D}^{\frac{1}{2}} f}W\abs{D} f \diff x  +  \int_\Rsp (W^2 \abs{D}^{\frac{1}{2}} f)(\abs{D}^{\frac{3}{2}} f) \diff x.
\end{align*}
	In this case, first observe that
	\begin{align*}
		W^{-1}\sqbrac{ W^2, \abs{D}^{\frac{1}{2}}} \abs{D}^{\frac{1}{2}} f &=- W\sqbrac{W^{-2},\abs{D}^\half}(W^2\abs{D}^\half f)
	\end{align*}
	which is bounded as follows
	\begin{align*}
		\norm[L^{2}]{W\sqbrac{W^{-2},\abs{D}^\half}(W^2\abs{D}^\half f)} &\lesssim \norm[L^{\infty}]{W}\norm[L^{2}]{W^{-3}W_{x}}\norm[L^{2}]{W^{2}\abs{D}^{\frac{1}{2}}f}
	\end{align*}
	Therefore,
	\begin{align*}
		\norm[2]{W \abs{D} f}^2 
		&\lesssim  \norm[L^{2}]{ W^{-1}\sqbrac{ W^2, \abs{D}^{\frac{1}{2}}} \abs{D}^{\frac{1}{2}} f} \norm[L^{2}]{ W\abs{D} f }  + \norm[L^{2}]{W^{2}\abs{D}^{\frac{1}{2}} f}\norm[\dot{H}^{\frac{3}{2}}]{f}\\
		&\lesssim_{W}  \norm[L^{2}]{W^{2}\abs{D}^{\frac{1}{2}}f} \norm[L^{2}]{ W\abs{D} f }  + \norm[L^{2}]{W^{2}\abs{D}^{\frac{1}{2}} f}\norm[\dot{H}^{\frac{3}{2}}]{f}\\
		&\lesssim_{W} \norm[L^{2}]{W^{2}\abs{D}^{\frac{1}{2}}f}^{2} + \norm[\dot{H}^{\frac{3}{2}}]{f}^{2},
	\end{align*}
	which combined with \eqref{ineq:firstwinfinitypapabshalf} implies the statement of the proposition in the case $W\in L^{\infty}(\mathbb{R})$.
\end{proof}

\begin{lem}\label{lem:infinitydecay}
Let $f:\mathbb{R}\rightarrow\mathbb{R}$, let $W_{1}, W_{2}: \mathbb{R} \rightarrow (0,\infty)$ be $C^1$ functions such that 
\begin{align*}
\abs{\partial_{x}\brac{W_{1}(x)W_{2}(x)}} &\lesssim W_{1}(x)^{2}
\end{align*}
and
\begin{align*}
W_{2}(x) &\lesssim \brac{1+x^{2}}^{\frac{1}{2}} W_{1}(x)
\end{align*}
for all $x\in\mathbb{R}$. Let $W = W_{1}^{\half}W_{2}^{\half}$ and let $f: \Rsp \to \Rsp$ be such that $f \in H^1_{loc}(\Rsp)$. If $W_{1}f \in L^{2}(\mathbb{R})$ and $W_{2}f_{x}\in L^{2}(\mathbb{R})$, then $Wf \in L^{\infty}(\mathbb{R})$ and 
\begin{align}\label{ineq:infinityweight}
	\norm[L^{\infty}]{Wf} &\lesssim  \norm[L^{2}]{W_{1}f} + \norm[L^{2}]{W_{2}f_{x}}.
\end{align}
Moreover, $W(x)f(x) \rightarrow 0$ as $\abs{x}\rightarrow \infty$.
\end{lem}

\begin{proof}
We first demonstrate the lemma with the additional condition that $f:\mathbb{R}\rightarrow\mathbb{R}$ is a compactly supported function. For any $x\in\mathbb{R}$, we have that
\begin{align*}
\abs{W(x)^{2}f(x)^{2}} &= \abs{\int_{-\infty}^{x} \partial_{y}\brac{W(y)^{2}f(y)^{2}}\diff y}\\
&\lesssim  \abs{\int_{-\infty}^{x}  \partial_{y}\brac{W(y)^{2}}f(y)^{2}\diff y} +  \abs{\int_{-\infty}^{x} W(y)^{2}f(y)f_{y}(y)\diff y}\\
&=  \abs{\int_{-\infty}^{x}  \partial_{y}\brac{W_{1}(y)W_{2}(y)}f(y)^{2}\diff y} +  \abs{\int_{-\infty}^{x} W_{1}(y)W_{2}(y)f(y)f_{y}(y)\diff y}\\
&\lesssim \norm[L^{2}]{W_{1}f}^{2} + \norm[L^{2}]{W_{1}f}\norm[L^{2}]{W_{2}f_{x}}.
\end{align*}
Hence, we obtain that
\begin{align*}
\norm[L^{\infty}]{Wf} &\lesssim  \norm[L^{2}]{W_{1}f} + \norm[L^{2}]{W_{2}f_{x}}.
\end{align*}
Now, let us consider the case where $f:\mathbb{R}\rightarrow\mathbb{R}$ is not compactly supported. Then, we have that $\chi_{\delta}f$ is compactly supported in $[-\delta^{-1},\delta^{-1}]$ and
\begin{align*}
\lim_{\delta\rightarrow 0}\norm[L^{2}]{W_{1}f-W_{1}\chi_{\delta}f} = 0.
\end{align*}
Moreover, since
\begin{align*}
\abs{W_{2}(x)\partial_{x}\chi_{\delta}(x)} &\lesssim \delta \abs{\brac{1+x^{2}}^{\frac{1}{2}} W_{1}(x)\partial_{x}\chi(\delta x)} \lesssim \abs{W_{1}(x)\partial_{x}\chi(\delta x)}
\end{align*}
we have that
\begin{align*}
\norm[L^{2}]{W_{2}f_{x}-W_{2}\brac{\chi_{\delta}f}_{x}} &\leq \norm[L^{2}]{W_{2}f_{x}-W_{2}\chi_{\delta}f_{x}} + \norm[L^{2}]{W_{1}\brac{\partial_{x}\chi(\delta x)}f}.  
\end{align*}
It follows that $\norm[L^{2}]{W_{1}\brac{\partial_{x}\chi(\delta x)}f}\rightarrow 0$ as $\delta \rightarrow 0$ from the Dominated Convergence Theorem since $W_{1}f\in L^{2}(\mathbb{R})$. Hence,
\begin{align*}
	\lim_{\delta\rightarrow 0}\norm[L^{2}]{W_{2}f_{x}-W_{2}\brac{\chi_{\delta}f}_{x}}  = 0. 
\end{align*}
By \eqref{ineq:infinityweight}, we have that $W\chi_{\delta}f$ is a Cauchy sequence in $L^{\infty}(\mathbb{R})$. As each of them tend to zero as $\abs{x}\rightarrow \infty$, we have that $Wf \in L^{\infty}(\mathbb{R})$ and moreover, it satisfies \eqref{ineq:infinityweight} and $W(x)f(x) \rightarrow 0$ as $\abs{x}\rightarrow \infty$.
\end{proof}


\bibliographystyle{amsplain}
\bibliography{/Users/siddhant/Desktop/Writing/Main.bib}

\end{document}